\documentclass[12pt,reqno]{amsart}
\usepackage{amscd,amssymb,amsmath,amsthm}
\usepackage[table]{xcolor}
\usepackage{ float}
\usepackage{multirow}
\usepackage{caption}
\usepackage{array, tabularx, caption, boldline}

\usepackage{booktabs}
\usepackage{cellspace}
\usepackage{cite}
\usepackage{tikz}
\usetikzlibrary {positioning}
\usetikzlibrary{arrows}
\usepackage{multirow}

\usepackage{amssymb,upref}
\usepackage[mathcal]{euscript}
\usepackage[cmtip,all]{xy}
\usepackage{mathtools}
\newtheorem{thm}[subsection]{Theorem}

\newtheorem{pro}[subsection]{Proposition}
\newtheorem{cor}[subsection]{Corollary}

\usepackage{tikz}
\usepackage{multirow}
\usetikzlibrary{arrows}
\newtheorem{rk}[subsection]{Remark}
\newtheorem{defn}[subsection]{Definition}

\numberwithin{equation}{section} 

\newcommand{\G}{{\mathcal G}}

\newcommand{\FF}{\mathbb{F}}

\newcommand{\bea}{\begin{eqnarray}}
\newcommand{\eea}{\end{eqnarray}}

  \def\G{\Gamma}

\begin{document}

\title[Derivation of Five-Dimensional Lotka-Voltera Algebra ]
{Derivation of Five-Dimensional Lotka-Voltera Algebra}
\author{Ahmad Alarafeen\hspace{0.5cm} Izzat Qaralleh\hspace{0.5cm}AZHANA AHMAD}

\address{Ahmad Alarafeen\\School of Mathematical Sciences\\, Universiti Sains Malaysia\\ 11800
USM, Penang, Malaysia}
\email{{\tt ahmadalarfeen@gmail.com}}
\address{Izzat Qaralleh\\
Department of Mathematics\\
Faculty of Science, Tafila Technical
University\\
Tafila, Jordan}
\email{{\tt izzat\_math@yahoo.com}}
\address{AZHANA AHMAD\\School of Mathematical Sciences\\, Universiti Sains Malaysia\\ 11800
USM, Penang, Malaysia}

\email{{\tt azhana@usm.my}}
\date{Received: xxxxxx; Revised: yyyyyy; Accepted: zzzzzz.
\newline \indent $^{*}$ Corresponding author}

\begin{abstract}
Lotka-Volterra (LV) algebras are generally applied in solving biological problems and  in examining the interactions among neighboring individuals. With reference to the methods applied by Gutierrez-Fernandez and Garcia in \cite{17}. this study examines the derivations of five-dimensional LV algebras. Before explicitly describing its derivation space, anti-symmetric matrices are used to define the LV algebra A.
 \vskip 0.3cm \noindent {\it
Mathematics Subject Classification}: 17D92, 17D99, 39A70, 47H10.\\
{\it Key words}: Evolution algebra; Volterra quadratic stochastic operator; Nilpotent; Isomorphism; Derivation.
\end{abstract}

\maketitle

\section{Introduction}
Many studies have examined non-associative algebras and their different classes, including baric, evolution, Bernstein, train, and stochastic algebras. While each of these classes has its own unique definition, they are all referred to as genetic algebras. For instance, the genetics of abstract algebras were investigated in  \cite{20} by referring to their formal language.

 Bernstein  \cite{21},\cite{20} introduced the concept of population genetics in their study of   evolution operators
 , which are often described by using quadratic stochastic operators (see  \cite{13}).\\
  Introduced in 1981 by Itoh [10], Lotka-Volterra (LV) algebras are formulations that are used to describe solutions and singularities and are  associated with quadratic differential equations [10,17].
They have been widely applied in different mathematical fields, including dynamic systems, Markov chains, graph theory, and population genetics \cite{10, 12}.\\

Normal Bernstein algebras, which are LV algebras associated with a zero anti-symmetric matrix, are essential in solving the classical Bernstein problem \cite{4,5,8,14}.
Many studies have described the structure of LV algebras \cite{1,2,3,6,7,10,16,17,24,25,26,27}. For instance,
 Holgate described how the derivations of genetic algebras, such as LV algebras, can be interpreted in biological terms \cite{9}.
 The derivations of LV algebras have been  classified up to three dimensions in \cite{7} and up to three and four dimensions in \cite{3}.
 Meanwhile, \cite{17},  proposed several novel approaches for examining  LV algebras.
For example, digraphs can be attached to LV algebras to classify their derivations of into four dimensions at most without requiring much  effort and while taking into account the  special cases in arbitrary dimensions.\\
By using the techniques introduced in \cite{17}, this paper attempts to classify the derivations of LV algebras up to five dimensions.
\section{Preliminaries}
This section describes all the definitions and theorems used in this paper to achieve the aforementioned objective.
 \begin{defn}
 A commutative (not necessarily associative) algebra $\mathbf{A}$
over a field $\mathbb{F}$ which characteristic is not equal to $2$ is called LV algebra if the natural  basis $B=\{e_{1},...,e_{n}\}$ is admitted such that $e_{i}e_{j}=\alpha_{ij}e_{i}+\alpha_{ji}e_{j}$, where $\alpha_{ij}$ and $\alpha_{ji}$ are in  field $\mathbf{A}$, whereas $a_{ii}=\frac{1}{2}$ and  $\alpha_{ij}+\alpha_{ji}=1$ for all $i,j=1, 2, ..., n.$
 \end{defn}
 The  structural constants of LV-algebra $\mathbf{A}$ are formulated in matrix form as  $\left(\alpha_{ij}\right)_{i,j=1}^n$ with respect to natural basis B.
 \begin{rk}
Each LV algebra $\mathbf{A}$ is baric in the sense that it admits a non-trivial homomorphism $\omega:\mathbf{A}\rightarrow \mathbb{F}$, which is  defined by $\omega\left(\sum_{i}\lambda_ie_i\right)=\sum_{i}\lambda_i$ for all $\lambda_i\in \mathbb{F}.$
 \end{rk}

 \begin{defn}
   Let  $(A,\circ)$ be an algebra with a linear
mapping
 $D : A \to A$
 such that
  $$D(u\circ v)=D(u)\circ v+u\circ D(v)$$
   for all $u,v\in A$ , derivation on algebra  $(A,\circ)$  $D\equiv 0$ is clearly a
derivation, which we call a \textit{trivial} one.
 \end{defn}
 We obtain the following notations on the basis of those described in \cite{17}. The hyper-plane unit and  barideal of $\mathbf{A}$ are defined as $H=\{x\in \mathbf{A}|\omega(x)=1\}$
and $N=ker \omega=\{x\in \mathbf{A} | \omega(x)=0\}$,  respectively.\\
supp(x) denotes the support of $x=\sum_{i}\lambda_{i}e_{i}$ with respect to natural basis $B$, and is induced when   the corresponding coefficient is nonzero, that is
$i\in supp(x)$ if and only if $\lambda_{i}\neq 0$. For simplification, we denote $\{1,2,...,n\}$ by $\mathbb{S}$. For each index $k,$ we define $\mathbb{S}_{k}$ as the set $\{i\in \mathbb{S}|e_{i}e_{k}=(e_{i}+e_{k})/2\}$, where $k\in \mathbb{S}$. In this case, we have, $k\in\mathbb{S}_{k}$, and index $i\in \mathbb{S}$ belongs to $\mathbb{S}_{k}$ if and only if $k\in \mathbb{S}_{i}$. Meanwhile, if $x_{1},...,x_{k}\in \mathbf{A}$,
 then $\left\langle x_{1},...,x_{k}\right\rangle$ is the vector subspace of $\mathbf{A}$ spanned by $x_{1},...,x_{k}$.\\
 Let $I$ be a non-empty subset of $\mathbb{S}$, and U= $\left\langle e_{i}|i\in I\right\rangle$.
  For each element $x=\sum_{k=1}^{n}\lambda_{k}e_{k}$ in $\mathbf{A}$, the $U$-component of $x$ is defined by $[x]_{U}=\sum_{i\in I}\lambda_{i}e_{i}$. If $U=\{e_{i}\}$, then we denote $[x]_{U}$ by $[x]_{i}$.  $L(\mathbf{A})$ represents the set of all linear mappings from $\mathbf{A}$ to $\mathbf{A}$.
In the following, $D$ is a derivation of $\mathbf{A}$, and  $\lambda_{tk}\in \FF$ such that	
\begin{equation}
D(e_{k})=\sum_{t=1}^{n}\lambda_{tk}e_{t}, \ \ k =\overline{1,n}.
\end{equation}

\begin{pro}\cite{17}
The following statements hold true:\\
\begin{itemize}
\item[(1)] If $D\in Der(\mathbf{A})$, then $Im D\subset N$ and $supp(D(e_{k}))\subset \mathbb{S}$ for all $k\in \mathbb{S}$.\\
\item[(2)] If $|\mathbb{S}_{k}|=1$ for all $k\in \mathbb{S}$, then $Der(\mathbf{A})=\{0\}$.\\
\item [(3)] If $i\in\mathbb{S}_{k}$, then supp$(D(e_{k}))\subseteq \mathbb{S}_{i}$.\\
\item[(4)] $supp(D(e_{k}))\subseteq \cap_{u\in\mathbb{S}_{k}}\mathbb{S}_{i}$, that is,  if $i,j\in supp(D(e_{k}))$, then\\ $e_{i}e_{j}=\frac{(e_{i}+e_{j})}{2}$ and $e_{i}e_{k}=\frac{(e_{i}+e_{k})}{2}$.\\
\end{itemize}
\end{pro}
 We define $\mathbb{N}_{k}=\cap_{i\in\mathbb{S}_{k}}\mathbb{S}_{i}$ for $k = 1, 2, ..., n.$, where index $j$ belongs to $\mathbb{N}_{k}$ if and only if $e_{i}e_{j}=(e_{i}+e_{j})/2$ for all $i\in \mathbb{S}_{k}$. We observe that $supp(D(e_{k}))\subseteq \mathbb{N}_{k}$ for all $j\in \mathbb{S}$.
\begin{thm}\cite{17}
  Let $e_{1}e_{2}=\alpha e_{1}+\alpha^{'}e_{2}$ where, $\alpha \neq \frac{1}{2}$, and define\\
\begin{center}
    $I = \mathbb{N}_{1}\cap \mathbb{N}_{2},\hspace{1.5cm}J = \mathbb{N}_{1}\setminus \mathbb{N}_{2},\hspace{1.5cm}K = \mathbb{N}_{2}\setminus \mathbb{N}_{1}$
\end{center}
\begin{center}
    U = $<e_{i}| i\in I >$ \hspace{1.5cm} V = $< e_{j} | j\in J>$\hspace{1.5cm} W = $< e_{k} | k\in K >$
\end{center}
Then,
\begin{center}
  $D(e_{1})= u + v, \hspace{0.6cm}and\hspace{0.6cm} D(e_{2})= u + w$.
\end{center}
  For some $u\in U$, we have $v\in V$ and $w\in W,$ where
 \begin{center}
  $\omega(u) = \omega(v) = \omega(w) = 0 $.
 \end{center}
Moreover,
\begin{center}
  $e_{j}e_{2}=\alpha e_{j}+\alpha^{'}e_{2},\hspace{0.6cm}and\hspace{0.6cm} e_{1}e_{k}=\alpha e_{1}+\alpha^{'}e_{k}$
\end{center}
for all $j\in supp(v)$ and $k\in supp(w)$.\\
\end{thm}
\begin{cor}\cite{17}
The following statements hold true:
\begin{itemize}
\item[(i)] Let $\mathbf{A}$ be similar to the previous theorem. If $|I|\leq 1$, then $u = 0.$ If $|J|\leq 1$, then $v = 0.$ If $|K|\leq 1$, then $w = 0.$
\item[(ii)] Let $e_{1}e_{2}=\alpha e_{1}+\alpha^{'}e_{2}$,  where $\alpha \neq\frac{1}{2}$.  Therefore, $[D(e_{1})]_{i}=[D(e_{2})]_{i}$ for all $i \in \mathbb{N}_{1}\cap \mathbb{N}_{2}.$
\item[(iii)] If i belongs to $\mathbb{N}_{k}$, then we have  $\mathbb{N}_{i}\subseteq \mathbb{N}_{k}$.
\item[(iv)] If $i\in\mathbb{N}_{k}$ and $k\in \mathbb{N}_{i}$, then $\mathbb{N}_{i}=\mathbb{N}_{k}.$

\end{itemize}
\end{cor}
Attached to the LV algebra
are constantly a graph and a weight graph that are connected and directed. Therefore, LV algebras will usually be presented in graphical form throughout the rest of this paper. The following properties  are satisfied by graph $\Gamma(\mathbf{A},B)=(\mathbb{S},E)$  :
  \begin{enumerate}
    \item $(i,j)\in E$ if and only if $(i,j)\notin E$ for all $i,j\in \mathbb{S}$,$i\neq j$, and
    \item $(i,i)\notin E$ for all $i \in\mathbb{S}$.
  \end{enumerate}

 The  weight graph attached to the LV algebra A relative to natural basis is represented by the triple $ \Gamma^{\tau}( \mathbf{A},\varphi)$ = $(\mathbb{S},E,\tau)$ where $\tau:E\longrightarrow F$ is given by $\tau(i,j)=\alpha_{ij}$. A unique
LV algebra is defined up to isomorphism by a weight graph. An LV algebra has equivalent graphs $(\mathbb{S},E,\tau)$ and $(\mathbb{S}^{'},E^{'},\tau)$  when $\mathbb{S}=\mathbb{S}^{'}$.  For every pair $(i,j)\in E,$ :if $(i,j)\in E^{'},$ then $\emph{v(i,j)}=\rho(i,j)$ or $\emph{v(j,i)}=1-\rho(i,j)$. If and only if two graphs are equivalent, then they can be are attached to the same LV algebra relative to the same natural
basis. We define those edges whose weights are equal and not equal to $\frac{1}{2}$, as :

$$
\begin{tikzpicture}
[->,>=stealth',shorten >=-5pt,auto,thick,every node/.style={circle,fill=blue!20,draw}]
  \node (n2) at (6,8)  {};

  \node (n1) at (8,8) {};

\path[every node/.style={font=\sffamily\small}]
 (n1) edge node [above] {} (n2);
  \tikzstyle{every node}=[font=\small\itshape]

\end{tikzpicture}
$$
Meanwhile, we define those edges that are not equal to$\neq\frac{1}{2}$  as:
$$
\begin{tikzpicture}
[->,>=stealth',shorten >=-5pt,auto,thick,every node/.style={circle,fill=blue!20,draw}]
  \node (n2) at (6,8)  {};

  \node (n1) at (8,8) {};

\path[every node/.style={font=\sffamily\small}][densely dashed,blue]
 (n1) edge node [above] {} (n2);
  \tikzstyle{every node}=[font=\small\itshape]

\end{tikzpicture}
$$
 The
degree of symmetry of an algebra may be assessed by using its weight graph. Therefore, we also take into consideration the subgraphs attached to algebra $\mathbf{A}.$
The following implications
hold (the elements of the basis are identified by the graphs on the basis of their subindices):\\
G.1:  If $\alpha\neq \beta$, then $3\not\in Supp(D(e_{2}))$.
$$
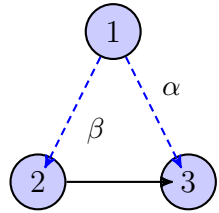
\begin{tikzpicture}
[->,>=stealth',shorten >=-5pt,auto,thick,every node/.style={circle,fill=blue!20,draw}]
  \node (n4) at (8,6)  {3};
  \node (n1) at (7,8) {1};
  \node (n3) at (6,6)  {2};
\path[every node/.style={font=\sffamily\small}]
 (n3) edge node [above] {} (n4);
  \tikzstyle{every node}=[font=\small\itshape]
 \draw (n1) -- (n4)  node [midway]{$\alpha$}[densely dashed,blue];
 \draw (n1) -- (n3)  node [midway]{$\beta$}[densely dashed,blue];
\end{tikzpicture}
$$
G.2:  $[D(e_{2})]_{1}=[D(e_{3})]_{1}$
$$
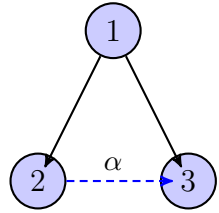
\begin{tikzpicture}
[->,>=stealth',shorten >=-5pt,auto,thick,every node/.style={circle,fill=blue!20,draw}]
  \node (n4) at (8,6)  {3};
  \node (n1) at (7,8) {1};
  \node (n3) at (6,6)  {2};
\path[every node/.style={font=\sffamily\small}]
 (n1) edge node [above] {} (n4)
 (n1) edge node [above] {} (n3);
  \tikzstyle{every node}=[font=\small\itshape]
 \draw (n3) -- (n4)  node [midway]{$\alpha$}[densely dashed,blue];
\end{tikzpicture}
$$
\begin{proof}
  G.1: Suppose that $3\in Supp(D(e_{2}))$, $e_{1}e_{2}=\alpha e_{1}+\alpha^{'}e_{2}$ for $\alpha\neq \frac{1}{2}$, we can obtain $3\in \mathbb{N}_{2}$ and $3\not\in \mathbb{N}_{1}$, . Therefore, we have $3\in \mathbb{N}_{2} \setminus (\mathbb{N}_{1}\cap \mathbb{N}_{2})$, according to the preivous theorem $e_{1}e_{3}=\alpha e_{1}+\alpha^{'}e_{3}$.\\
  \\G.2: Suppose  that G.2 holds. If $1\notin supp(D(e_{2}))\cap$
$supp(D(e_{3}))$, then  $[D(e_{2})]_{1} = [D(e_{3})]_{1}$ because both vanish. Now let $1 \notin supp(D(e_{2}))\cup supp(D(e_{3}))$, and $1\in supp(D(e3))$.
  We observe that $1 \notin \mathbb{N}_{3}\setminus (\mathbb{N}_{3}\cap \mathbb{N}_{2})$ , and according to Theorem 1, we obtain $[e_{2}e_{k}]_{k}$ = $(1- \alpha ) e_{k}$
  for all k $ \mathbb{N}_{3}\ (\mathbb{N}_{3}\cap \mathbb{N}_{2})\cap supp(D(e_{3}))$.
  Therefore, $1\in \mathbb{N}_{2}\cap \mathbb{N}_{3}$ according to the results of Corollary 2.6(ii), where $[D(e_{2})]_{1} = [D(e_{3})]_{1}$.

\end{proof}
\begin{rk}
fully describes the derivation of 3D LV algebra, whereas \cite{7},  presents a new technique for describing the 3D and 4D LV algebra. To describe the derivation of 5D LV algebra, we use the technique proposed in \cite{17} .
\end{rk}
\section{Main Results}
Let $\mathbf{A}$ be a 5D LV algebra over a field whose characteristic is unequal to 2, and  let D is a derivation of $\mathbf{A}$. With reference to a suitable natural basis, A is represented by a weighted graph $\gamma(\mathbf{A})$ , which  denotes the number of edges whose weight is not equal to $\frac{1}{2}$. If $\gamma(\mathbf{A})=0,$ then  $Der(\mathbf{A})$ represents the  linear mapping $\emph{f}$ from $\mathbf{A}$ to $\mathbf{A}$ such that $Im\emph{(f)}\subset N$. We present below some cases of 5D $\gamma(\mathbf{A})$ .
\begin{thm}
  Let $\textbf{A}$ be a 5D LV algebra, and in the following table summarizes its derivations:

\begin{table}[H]
  \centering
  \begin{tabular}{| m{1cm} | m{4cm}| m{10cm}|}
   \hline\hline
     Name  & Graph of  & $Der(\mathbf{A})$   \\
     \hline
   $M_{1}$ & \begin{tikzpicture}\label{35}
  [->,>=stealth',shorten >=-1pt,auto,thick,every node/.style={circle,fill=blue!5,draw}]
  \node (n5) at (8,6) {5};
  \node (n2) at (6,8)  {2};
  \node (n4) at (7,4.6)  {4};
  \node (n1) at (8,8) {1};
  \node (n3) at (6,6)  {3};
  \path[every node/.style={font=\sffamily\small}]
 (n1) edge node [above] {} (n3)
 (n1) edge node [above] {} (n4)
 (n1) edge node [above] {} (n5)
 (n2) edge node [above] {} (n3)
 (n2) edge node [above] {} (n4)
 (n2) edge node [above] {} (n5)
 (n3) edge node [above] {} (n5)
  (n3) edge node [above] {} (n4)
  (n1) edge node [above] {} (n2)
 (n4) edge node [above] {} (n5);

 \end{tikzpicture}& $\{f\in L(\mathbf{A})|Im(f)\subset N\}$\\\hline
  $M_{2}$ & \begin{tikzpicture}\label{35}
  [->,>=stealth',shorten >=-1pt,auto,thick,every node/.style={circle,fill=blue!5,draw}]
  \node (n5) at (8,6) {5};
  \node (n2) at (6,8)  {2};
  \node (n4) at (7,4.6)  {4};
  \node (n1) at (8,8) {1};
  \node (n3) at (6,6)  {3};
  \path[every node/.style={font=\sffamily\small}]
 (n1) edge node [above] {} (n3)
 (n1) edge node [above] {} (n4)
 (n1) edge node [above] {} (n5)
 (n2) edge node [above] {} (n3)
 (n2) edge node [above] {} (n4)
 (n2) edge node [above] {} (n5)
  (n3) edge node [above] {} (n4)
    (n1) edge node [above] {} (n5)
    (n3) edge node [above] {} (n5)
 (n4) edge node [above] {} (n5);
  \draw[densely dashed,blue]  (n1) -- (n2)node [midway]{};
 \end{tikzpicture}& $\{f\in L(\mathbf{A})|Im(f)\subset <e_{3}-e_{4}, e_{3}-e_{5}>$ and $f(e_{1}) = f(e_{2})$\}\\\hline
 $M_{3}$ & \begin{tikzpicture}\label{36}
  [->,>=stealth',shorten >=-5pt,auto,thick,every node/.style={circle,fill=blue!20,draw}]
  \node (n5) at (8,6) {5};
  \node (n2) at (6,8)  {2};
  \node (n4) at (7,4.6)  {4};
  \node (n1) at (8,8) {1};
  \node (n3) at (6,6)  {3};
\path[every node/.style={font=\sffamily\small}]
 (n1) edge node [above] {} (n2)
 (n1) edge node [above] {} (n3)
 (n1) edge node [above] {} (n4)
 (n2) edge node [above] {} (n3)
 (n2) edge node [above] {} (n4)
 (n2) edge node [above] {} (n5)
 (n3) edge node [above] {} (n4)
 (n4) edge node [above] {} (n5);
 \tikzstyle{every node}=[font=\small\itshape]
 \draw[densely dashed,blue]  (n1) -- (n5)node [midway]{$\alpha$};
 \draw[densely dashed,blue]  (n3) -- (n5)node [midway]{$\alpha$};
\end{tikzpicture} & $\{f\in L(\mathbf{A})|Im(f)\subset <e_{2}-e_{4}>$ and $f(e_{3})=f(e_{5})= f(e_{1})\}$\\\hline

 $M_{4}$ & \begin{tikzpicture}\label{35}
[->,>=stealth',shorten >=-5pt,auto,thick,every node/.style={circle,fill=blue!20,draw}]
  \node (n5) at (8,6) {5};
  \node (n2) at (6,8)  {2};
  \node (n4) at (7,4.6)  {4};
  \node (n1) at (8,8) {1};
  \node (n3) at (6,6)  {3};
\path[every node/.style={font=\sffamily\small}]
 (n1) edge node [above] {} (n4)
 (n2) edge node [above] {} (n3)
 (n2) edge node [above] {} (n4)
 (n2) edge node [above] {} (n5)
 (n3) edge node [above] {} (n4)
 (n3) edge node [above] {} (n5)
 (n4) edge node [above] {} (n5);
 \tikzstyle{every node}=[font=\small\itshape]
 \draw[densely dashed,blue]  (n1) -- (n2)node [midway]{$\alpha$};
 \draw[densely dashed,blue]  (n1) -- (n3)node [midway]{$\alpha$};
 \draw[densely dashed,blue]  (n1) -- (n5)node [midway]{$\gamma$};
\end{tikzpicture} & $\{f\in L(\mathbf{A})| e_{1}, e_{5},e_{4}\in ker(f)$ and $Im(f)\subset< e_{2} - e_{3}>\}$ \\\hline

 $M_{5}$ & \begin{tikzpicture}\label{35}
[->,>=stealth',shorten >=-5pt,auto,thick,every node/.style={circle,fill=blue!20,draw}]
  \node (n5) at (8,6) {5};
  \node (n2) at (6,8)  {2};
  \node (n4) at (7,4.6)  {4};
  \node (n1) at (8,8) {1};
  \node (n3) at (6,6)  {3};
\path[every node/.style={font=\sffamily\small}]
 (n1) edge node [above] {} (n4)
 (n2) edge node [above] {} (n3)
 (n2) edge node [above] {} (n4)
 (n2) edge node [above] {} (n5)
 (n3) edge node [above] {} (n4)
 (n3) edge node [above] {} (n5)
 (n4) edge node [above] {} (n5);
 \tikzstyle{every node}=[font=\small\itshape]
 \draw[densely dashed,blue]  (n1) -- (n2)node [midway]{$\alpha$};
 \draw[densely dashed,blue]  (n1) -- (n3)node [midway]{$\alpha$};
 \draw[densely dashed,blue]  (n1) -- (n5)node [midway]{$\alpha$};
\end{tikzpicture} & $\{f\in L(\mathbf{A})|Im(\emph{f})\subset<e_{2}- e_{3},e_{2}-e_{5} >$ and $e_{1}, e_{4}\in$ $ker(\emph{f})\}$ \\\hline
\end{tabular}
\end{table}
\begin{table}[H]
  \centering
   \begin{tabular}{| m{1cm} | m{4cm}| m{10cm} |}
   \hline\hline
     Name  & Graph of  & $Der(\mathbf{A})$   \\
     \hline

 $M_{6}$ & \begin{tikzpicture}\label{35}
[->,>=stealth',shorten >=-5pt,auto,thick,every node/.style={circle,fill=blue!20,draw}]
   \node (n5) at (8,6) {5};
  \node (n2) at (6,8)  {2};
  \node (n4) at (7,4.6)  {4};
  \node (n1) at (8,8) {1};
  \node (n3) at (6,6)  {3};
\path[every node/.style={font=\sffamily\small}]
 (n1) edge node [above] {} (n3)
 (n1) edge node [above] {} (n4)
 (n2) edge node [above] {} (n3)
 (n2) edge node [above] {} (n4)
 (n2) edge node [above] {} (n5)
 (n3) edge node [above] {} (n5)
 (n4) edge node [above] {} (n5);
\tikzstyle{every node}=[font=\small\itshape]
 \draw [densely dashed,blue](n1) -- (n2) node [midway]{$\alpha$}   ;
 \draw [densely dashed,blue](n1) -- (n5) node [midway]{$\alpha$}   ;
 \draw [densely dashed,blue](n3) -- (n4) node [midway]{$\gamma$} ;
\end{tikzpicture}& $\{f\in L(\mathbf{A})|Im(f)\subset<e_{2}-e_{5}>$ and $e_{1},e_{3}, e_{4}\in ker(\emph{f})\}$\\\hline

$M_{7}$ & \begin{tikzpicture}\label{35}
[->,>=stealth',shorten >=-5pt,auto,thick,every node/.style={circle,fill=blue!20,draw}]
   \node (n5) at (8,6) {5};
  \node (n2) at (6,8)  {2};
  \node (n4) at (7,4.6)  {4};
  \node (n1) at (8,8) {1};
  \node (n3) at (6,6)  {3};
\path[every node/.style={font=\sffamily\small}]
 (n1) edge node [above] {} (n3)
 (n1) edge node [above] {} (n4)
 (n2) edge node [above] {} (n3)
 (n2) edge node [above] {} (n4)
 (n3) edge node [above] {} (n4)
 (n3) edge node [above] {} (n5)
 (n4) edge node [above] {} (n5);
 \tikzstyle{every node}=[font=\small\itshape]
 \draw[densely dashed,blue]  (n1) -- (n2)node [midway]{$\alpha$};
 \draw[densely dashed,blue]  (n1) -- (n5)node [midway]{$\gamma$};
 \draw[densely dashed,blue]  (n2) -- (n5)node [midway]{$\beta$};
\end{tikzpicture} & $\{f\in L(\mathbf{A})|Im(f)\subset< e_{3} - e_{4} >$ and $f(e_{1}) = f(e_{2}) = f(e_{5})\}$\\\hline
$M_{8}$&\begin{tikzpicture}\label{35}
[->,>=stealth',shorten >=-5pt,auto,thick,every node/.style={circle,fill=blue!20,draw}]
  \node (n5) at (8,6) {5};
  \node (n2) at (6,8)  {2};
  \node (n4) at (7,4.6)  {4};
  \node (n1) at (8,8) {1};
  \node (n3) at (6,6)  {3};
\path[every node/.style={font=\sffamily\small}]
 (n1) edge node [above] {} (n3)
 (n1) edge node [above] {} (n4)
 (n1) edge node [above] {} (n5)
 (n2) edge node [above] {} (n4)
 (n2) edge node [above] {} (n5)
 (n4) edge node [above] {} (n5);
 \draw[densely dashed,blue]  (n1) -- (n2);
 \draw[densely dashed,blue]  (n2) -- (n3);
 \draw[densely dashed,blue]  (n3) -- (n4)node[midway]{$\alpha$};
 \draw[densely dashed,blue]  (n3) -- (n5)node[midway]{$\alpha$};
\end{tikzpicture}&$\{f\in L(\mathbf{A})| e_{1}, e_{2},e_{3}\in ker(f)$ and $Im(f)\subset< e_{4} - e_{5}>\}$ \\\hline
 $M_{9}$ & \begin{tikzpicture}\label{35}
[->,>=stealth',shorten >=-5pt,auto,thick,every node/.style={circle,fill=blue!20,draw}]
  \node (n5) at (8,6) {5};
  \node (n2) at (6,8)  {2};
  \node (n4) at (7,4.6)  {4};
  \node (n1) at (8,8) {1};
  \node (n3) at (6,6)  {3};
\path[every node/.style={font=\sffamily\small}]
 (n1) edge node [above] {} (n4)
 (n2) edge node [above] {} (n4)
 (n2) edge node [above] {} (n5)
 (n3) edge node [above] {} (n4)
 (n1) edge node [above] {} (n3)
 (n4) edge node [above] {} (n5);
 \tikzstyle{every node}=[font=\small\itshape]
 \draw[densely dashed,blue]  (n1) -- (n2)node [midway]{$\alpha$};
 \draw[densely dashed,blue]  (n2) -- (n3)node [midway]{$\alpha$};
 \draw[densely dashed,blue]  (n3) -- (n5)node [midway]{$\beta$};
 \draw[densely dashed,blue]  (n1) -- (n5)node [midway]{$\beta$};
\end{tikzpicture} & $\{f\in L(\mathbf{A})|f(e_{1}), f(e_{3})\in < e_{1} - e_{3} >$, $f(e_{2}), f(e_{5})\in < e_{2} - e_{5} >$ and
  $e_{4}\in ker(f)\}$\\\hline
 $M_{10}$ & \begin{tikzpicture}\label{35}
[->,>=stealth',shorten >=-5pt,auto,thick,every node/.style={circle,fill=blue!20,draw}]
  \node (n5) at (8,6) {5};
  \node (n2) at (6,8)  {2};
  \node (n4) at (7,4.6)  {4};
  \node (n1) at (8,8) {1};
  \node (n3) at (6,6)  {3};
\path[every node/.style={font=\sffamily\small}]
 (n2) edge node [above] {} (n3)
 (n2) edge node [above] {} (n4)
 (n2) edge node [above] {} (n5)
 (n3) edge node [above] {} (n4)
 (n3) edge node [above] {} (n5)
 (n4) edge node [above] {} (n5);
 \tikzstyle{every node}=[font=\small\itshape]
 \draw[densely dashed,blue]  (n1) -- (n2)node [midway]{$\alpha$};
 \draw[densely dashed,blue]  (n1) -- (n5)node [midway]{$\alpha$};
 \draw[densely dashed,blue]  (n1) -- (n3)node [midway]{$\alpha$};
 \draw[densely dashed,blue]  (n1) -- (n4)node [midway]{$\alpha$};
\end{tikzpicture} & $\{f\in L(\mathbf{A})| e_{1}\in ker(f)$ and $Im(f)\subset < e_{2} - e_{3},  e_{2} - e_{4},  e_{2} - e_{5}>$ \}\\\hline
  \end{tabular}
\end{table}
\begin{table}[H]
  \centering
   \begin{tabular}{| m{1cm} | m{4cm}| m{10cm} |}
   \hline\hline
     Name  & Graph of  & $Der(\mathbf{A})$   \\
     \hline
 $M_{11}$ & \begin{tikzpicture}\label{35}
 [->,>=stealth',shorten >=-5pt,auto,thick,every node/.style={circle,fill=blue!20,draw}]
  \node (n5) at (8,6) {5};
  \node (n2) at (6,8)  {2};
  \node (n4) at (7,4.6)  {4};
  \node (n1) at (8,8) {1};
  \node (n3) at (6,6)  {3};
\path[every node/.style={font=\sffamily\small}]
 (n2) edge node [above] {} (n3)
 (n2) edge node [above] {} (n4)
 (n2) edge node [above] {} (n5)
 (n3) edge node [above] {} (n4)
 (n3) edge node [above] {} (n5)
 (n4) edge node [above] {} (n5);
 \tikzstyle{every node}=[font=\small\itshape]
  \draw[densely dashed,blue]  (n1) -- (n2)node [midway]{$\alpha$};
 \draw[densely dashed,blue]  (n1) -- (n5)node [midway]{$\alpha$};
 \draw[densely dashed,blue]  (n1) -- (n3)node [midway]{$\alpha$};
 \draw[densely dashed,blue]  (n1) -- (n4)node[midway]{$\delta$};
\end{tikzpicture} & \{ $f\in L(\mathbf{A})| Im(f)\subset < e_{2} $ - $e_{3},  e_{2} - e_{5} >$ and $e_{4}, e_{1}\in ker(f)$ \}\\\hline

 $M_{12}$& $\begin{tikzpicture}\label{35}
[->,>=stealth',shorten >=-5pt,auto,thick,every node/.style={circle,fill=blue!20,draw}]
  \node (n5) at (8,6) {5};
  \node (n2) at (6,8)  {2};
  \node (n4) at (7,4.6)  {4};
  \node (n1) at (8,8) {1};
  \node (n3) at (6,6)  {3};
\path[every node/.style={font=\sffamily\small}]
 (n2) edge node [above] {} (n3)
 (n2) edge node [above] {} (n4)
 (n2) edge node [above] {} (n5)
 (n3) edge node [above] {} (n4)
 (n3) edge node [above] {} (n5)
 (n4) edge node [above] {} (n5);
 \tikzstyle{every node}=[font=\small\itshape]
  \draw[densely dashed,blue]  (n1) -- (n2)node [midway]{$\alpha$};
 \draw[densely dashed,blue]  (n1) -- (n5)node [midway]{$\alpha$};
 \draw[densely dashed,blue]  (n1) -- (n3)node [midway]{$\delta$};
 \draw[densely dashed,blue]  (n1) -- (n4)node [midway]{$\delta$};
\end{tikzpicture}$ & \{$f\in L(\mathbf{A})| f(e_{2}), f(e_{5})\in < e_{2} - e_{5}>$, $f(e_{3}), f(e_{4})\in < e_{3} - e_{4}>$ and
$e_{1}\in ker(f)$\}\\\hline
$M_{13}$ & $\begin{tikzpicture}\label{35}
[->,>=stealth',shorten >=-5pt,auto,thick,every node/.style={circle,fill=blue!20,draw}]
  \node (n5) at (8,6) {5};
  \node (n2) at (6,8)  {2};
  \node (n4) at (7,4.6)  {4};
  \node (n1) at (8,8) {1};
  \node (n3) at (6,6)  {3};
\path[every node/.style={font=\sffamily\small}]
 (n2) edge node [above] {} (n3)
 (n2) edge node [above] {} (n4)
 (n2) edge node [above] {} (n5)
 (n3) edge node [above] {} (n4)
 (n3) edge node [above] {} (n5)
 (n4) edge node [above] {} (n5);
 \tikzstyle{every node}=[font=\small\itshape]
  \draw[densely dashed,blue]  (n1) -- (n2)node [midway]{$\alpha$};
 \draw[densely dashed,blue]  (n1) -- (n5)node [midway]{$\alpha$};
 \draw[densely dashed,blue]  (n1) -- (n3)node [midway]{$\gamma$};
 \draw[densely dashed,blue]  (n1) -- (n4)node [midway]{$\delta$};
\end{tikzpicture}$ & $\{f\in L(\mathbf{A})| Im(f)\subset < e_{2} - e_{5}>$ and $e_{1}, e_{4}, e_{3}\in ker(f) \}$\\\hline
$M_{14}$ & \begin{tikzpicture}\label{1}
[->,>=stealth',shorten >=-5pt,auto,thick,every node/.style={circle,fill=blue!10,draw}]
  \node (n5) at (8,6) {5};
  \node (n2) at (6,8)  {2};
  \node (n4) at (7,4.6)  {4};
  \node (n1) at (8,8) {1};
  \node (n3) at (6,6)  {3};
\path[every node/.style={font=\sffamily\small}]
 (n2) edge node [above] {} (n3)
 (n2) edge node [above] {} (n4)
 (n2) edge node [above] {} (n5)
 (n3) edge node [above] {} (n4)
 (n4) edge node [above] {} (n5);
 \tikzstyle{every node}=[font=\small\itshape]
 \draw[densely dashed,blue]  (n1) -- (n2)node [midway]{$\alpha$};
 \draw[densely dashed,blue]  (n1) -- (n3);
 \draw[densely dashed,blue]  (n1) -- (n4)node [midway]{$\alpha$};
 \draw[densely dashed,blue]  (n1) -- (n5);
 \draw[densely dashed,blue]  (n3) -- (n5);
\end{tikzpicture} & $\{f\in L(\mathbf{A})| Im(f)\subset <e_{2} - e_{4}>$ and $e_{1} \in ker(f)\}$\\\hline
$M_{15}$ & \begin{tikzpicture}\label{1}
[->,>=stealth',shorten >=-5pt,auto,thick,every node/.style={circle,fill=blue!10,draw}]
 \node (n5) at (8,6) {5};
  \node (n2) at (6,8)  {2};
  \node (n4) at (7,4.6)  {4};
  \node (n1) at (8,8) {1};
  \node (n3) at (6,6)  {3};
\path[every node/.style={font=\sffamily\small}]
 (n1) edge node [above] {} (n4)
 (n2) edge node [above] {} (n4)
 (n2) edge node [above] {} (n5)
 (n3) edge node [above] {} (n4)
 (n4) edge node [above] {} (n5);
 \tikzstyle{every node}=[font=\small\itshape]
 \draw[densely dashed,blue] (n1) -- (n2)node [midway]{$\alpha$};
 \draw[densely dashed,blue] (n2) -- (n3);
 \draw[densely dashed,blue] (n3) -- (n5);
 \draw[densely dashed,blue] (n1) -- (n5)node [midway]{$\alpha$};
 \draw[densely dashed,blue] (n1) -- (n3);
\end{tikzpicture} & $\{f\in L(\mathbf{A})| Im(f)\subset < e_{2} - e_{5}>$ and $e_{1}, e_{3}, e_{4}\in ker(f)\}$\\\hline
 \end{tabular}
\end{table}
\begin{table}[H]
  \centering
   \begin{tabular}{| m{1cm} | m{4cm}| m{10cm} |}
   \hline\hline
     Name  & Graph of  & $Der(\mathbf{A})$   \\
     \hline
$M_{16}$ & \begin{tikzpicture}\label{1}
[->,>=stealth',shorten >=-5pt,auto,thick,every node/.style={circle,fill=blue!10,draw}]
 \node (n5) at (8,6) {5};
  \node (n2) at (6,8)  {2};
  \node (n4) at (7,4.6)  {4};
  \node (n1) at (8,8) {1};
  \node (n3) at (6,6)  {3};
\path[every node/.style={font=\sffamily\small}]
 (n1) edge node [above] {} (n3)
 (n2) edge node [above] {} (n4)
 (n2) edge node [above] {} (n5)
 (n3) edge node [above] {} (n4)
 (n4) edge node [above] {} (n5);
 \draw[densely dashed,blue] (n1) -- (n2)node [midway]{$\alpha$};
 \draw[densely dashed,blue] (n2) -- (n3);
 \draw[densely dashed,blue] (n3) -- (n5);
 \draw[densely dashed,blue] (n1) -- (n5)node [midway]{$\alpha$};
 \draw[densely dashed,blue] (n1) -- (n4);
\end{tikzpicture} & $\{f\in L(\mathbf{A})| Im(f)\subset < e_{2} - e_{5}>$ and $e_{1}, e_{3}, e_{4}\in ker(f)\}$\\\hline

$M_{17}$ & \begin{tikzpicture}\label{121}
[->,>=stealth',shorten >=-5pt,auto,thick,every node/.style={circle,fill=blue!20,draw}]
  \node (n5) at (8,6) {5};
  \node (n2) at (6,8)  {2};
  \node (n4) at (7,4.6)  {4};
  \node (n1) at (8,8) {1};
  \node (n3) at (6,6)  {3};
\path[every node/.style={font=\sffamily\small}]
 (n2) edge node [above] {} (n4)
 (n4) edge node [above] {} (n5)
 (n2) edge node [above] {} (n5)
 (n3) edge node [above] {} (n4);
 \tikzstyle{every node}=[font=\small\itshape]
 \draw[densely dashed,blue] (n1) -- (n2)node [midway]{$\alpha$};
 \draw[densely dashed,blue] (n2) -- (n3);
 \draw[densely dashed,blue] (n3) -- (n5);
 \draw[densely dashed,blue] (n1) -- (n5)node [midway]{$\alpha$};
 \draw[densely dashed,blue] (n1) -- (n3);
 \draw[densely dashed,blue] (n1) -- (n4);
\end{tikzpicture} & $\{f\in L(\mathbf{A})| Im(f)\subset < e_{2} - e_{5}>$ and  $e_{1}, e_{3}, e_{4}\in ker(f)\}$\\\hline
$M_{18}$ & \begin{tikzpicture}\label{121}
[->,>=stealth',shorten >=-5pt,auto,thick,every node/.style={circle,fill=blue!20,draw}]
   \node (n5) at (8,6) {5};
  \node (n2) at (6,8)  {2};
  \node (n4) at (7,4.6)  {4};
  \node (n1) at (8,8) {1};
  \node (n3) at (6,6)  {3};
\path[every node/.style={font=\sffamily\small}]
 (n2) edge node [above] {} (n5)
 (n3) edge node [above] {} (n4)
 (n1) edge node [above] {} (n3)
 (n4) edge node [above] {} (n5);
 \tikzstyle{every node}=[font=\small\itshape]
 \draw[densely dashed,blue] (n1) -- (n2)node [midway]{$\alpha$};
 \draw[densely dashed,blue] (n2) -- (n3);
 \draw[densely dashed,blue] (n3) -- (n5);
 \draw[densely dashed,blue] (n1) -- (n5)node [midway]{$\alpha$};
 \draw[densely dashed,blue] (n1) -- (n4);
 \draw[densely dashed,blue] (n2) -- (n4);
\end{tikzpicture} & $\{f\in L(\mathbf{A})| Im(f)\subset < e_{2} - e_{5}>$ and $e_{1}, e_{3}, e_{4}\in ker(f)\}$\\\hline
$M_{19}$ & \begin{tikzpicture}\label{32}
[->,>=stealth',shorten >=-5pt,auto,thick,every node/.style={circle,fill=blue!20,draw}]
  \node (n5) at (8,6) {5};
  \node (n2) at (6,8)  {2};
  \node (n4) at (7,4.6)  {4};
  \node (n1) at (8,8) {1};
  \node (n3) at (6,6)  {3};
\path[every node/.style={font=\sffamily\small}]
 (n2) edge node [above] {} (n4)
 (n1) edge node [above] {} (n3)
 (n2) edge node [above] {} (n5);
 \tikzstyle{every node}=[font=\small\itshape]
 \draw[densely dashed,blue] (n1) -- (n2);
 \draw[densely dashed,blue] (n2) -- (n3);
 \draw[densely dashed,blue] (n3) -- (n4);
 \draw[densely dashed,blue] (n4) -- (n5);
 \draw[densely dashed,blue] (n1) -- (n4);
 \draw[densely dashed,blue] (n3) -- (n5)node [midway]{$\alpha$};
 \draw[densely dashed,blue] (n1) -- (n5)node [midway]{$\alpha$};
\end{tikzpicture} & $\{f\in L(\mathbf{A})| Im(f)\subset < e_{1} - e_{3}>$ and $e_{2}, e_{4}, e_{5}\in ker(f)\}$\\\hline
 $M_{20}$ & \begin{tikzpicture}\label{100}
[->,>=stealth',shorten >=-5pt,auto,thick,every node/.style={circle,fill=blue!20,draw}]
  \node (n5) at (8,6) {5};
  \node (n2) at (6,8)  {2};
  \node (n4) at (7,4.6)  {4};
  \node (n1) at (8,8) {1};
  \node (n3) at (6,6)  {3};
\path[every node/.style={font=\sffamily\small}]
 (n1) edge node [above] {} (n3)
 (n2) edge node [above] {} (n5);
 \tikzstyle{every node}=[font=\small\itshape]
 \draw[densely dashed,blue] (n1) -- (n2);
 \draw[densely dashed,blue] (n2) -- (n4)node [midway]{$\delta$};
 \draw[densely dashed,blue] (n1) -- (n4)node [midway]{$\alpha$};
 \draw[densely dashed,blue] (n1) -- (n5);
 \draw[densely dashed,blue] (n2) -- (n3);
 \draw[densely dashed,blue] (n3) -- (n4)node [midway]{$\beta$};
 \draw[densely dashed,blue] (n3) -- (n5);
 \draw[densely dashed,blue] (n4) -- (n5)node [midway]{$\delta$};
\end{tikzpicture} & $\{f\in L(\mathbf{A})| Im(f)\subset < e_{1} - e_{3}>$ and $e_{2}, e_{5}, e_{4}\in ker(f)\}$\\\hline
 \end{tabular}
\end{table}
\begin{table}[H]
  \centering
  \begin{tabular}{|c|c|c|}
    \hline
$M_{21}$ & \begin{tikzpicture}\label{100}
[->,>=stealth',shorten >=-5pt,auto,thick,every node/.style={circle,fill=blue!20,draw}]
  \node (n5) at (8,6) {5};
  \node (n2) at (6,8)  {2};
  \node (n4) at (7,4.6)  {4};
  \node (n1) at (8,8) {1};
  \node (n3) at (6,6)  {3};
\path[every node/.style={font=\sffamily\small}]
 (n1) edge node [above] {} (n3)
 (n2) edge node [above] {} (n5);
 \tikzstyle{every node}=[font=\small\itshape]
 \draw[densely dashed,blue] (n1) -- (n2);
 \draw[densely dashed,blue] (n2) -- (n4)node [midway]{$\delta$};
 \draw[densely dashed,blue] (n1) -- (n4)node [midway]{$\alpha$};
 \draw[densely dashed,blue] (n1) -- (n5);
 \draw[densely dashed,blue] (n2) -- (n3);
 \draw[densely dashed,blue] (n3) -- (n4)node [midway]{$\alpha$};
 \draw[densely dashed,blue] (n3) -- (n5);
 \draw[densely dashed,blue] (n4) -- (n5)node [midway]{$\delta$};
\end{tikzpicture} & $\{f\in L(\mathbf{A})| f(e_{1}), f(e_{3})\in< e_{1} - e_{3}>$ and $f(e_{2}), f(e_{5})\in < e_{2} - e_{5}>\}$\\\hline

  $M_{22}$& \begin{tikzpicture}\label{35}
[->,>=stealth',shorten >=-5pt,auto,thick,every node/.style={circle,fill=blue!20,draw}]
  \node (n5) at (8,6) {5};
  \node (n2) at (6,8)  {2};
  \node (n4) at (7,4.6)  {4};
  \node (n1) at (8,8) {1};
  \node (n3) at (6,6)  {3};
\path[every node/.style={font=\sffamily\small}]
 (n2) edge node [above] {} (n4);
 \tikzstyle{every node}=[font=\small\itshape]
 \draw[densely dashed,blue] (n2) -- (n3)node [midway]{$\alpha$};
 \draw[densely dashed,blue] (n3) -- (n4)node [midway]{$\alpha$};
 \draw[densely dashed,blue] (n1) -- (n2);
 \draw[densely dashed,blue] (n1) -- (n3);
 \draw[densely dashed,blue] (n1) -- (n4);
 \draw[densely dashed,blue] (n1) -- (n5);
 \draw[densely dashed,blue] (n3) -- (n5);
 \draw[densely dashed,blue] (n4) -- (n5);
 \draw[densely dashed,blue] (n2) -- (n5);
\end{tikzpicture}& $\{f\in L(\mathbf{A})| Im(f)\subset < e_{2} - e_{4}>$ and $e_{1}, e_{3}, e_{5}\in ker(f)\}$\\\hline
\end{tabular}
\end{table}
\end{thm}
\begin{proof}
Depending on the order of the set, the proof is divided into the following cases:

\begin{itemize}

\item[Case 1:] If $\gamma(\mathbf{A})=1$, then
 $\mathbf{A}$ can be graphed as follows up to equivalence  and isomorphism :
$$
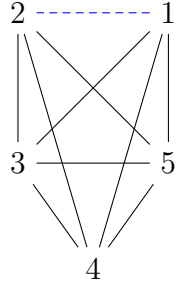
\begin{tikzpicture}\label{35}
[->,>=stealth',shorten >=-5pt,auto,thick,every node/.style={circle,fill=blue!20,draw}]
  \node (n5) at (8,6) {5};
  \node (n2) at (6,8)  {2};
  \node (n4) at (7,4.6)  {4};
  \node (n1) at (8,8) {1};
  \node (n3) at (6,6)  {3};
\path[every node/.style={font=\sffamily\small}]
 (n1) edge node [above] {} (n3)
 (n1) edge node [above] {} (n4)
 (n1) edge node [above] {} (n5)
 (n2) edge node [above] {} (n3)
 (n2) edge node [above] {} (n4)
 (n2) edge node [above] {} (n5)
 (n3) edge node [above] {} (n5)
  (n3) edge node [above] {} (n4)
 (n4) edge node [above] {} (n5);
 \draw[densely dashed,blue] (n1) -- (n2);
\end{tikzpicture}
$$
In this graph, we have $\mathbb{S}_{1}=\{1,3,4,5\}$,  $\mathbb{S}_{2}=\{2,3,4,5\}=S_{5}$, and   $\mathbb{S}_{k}=\mathbb{S}$ , where $k=3, 4, 5.$  From these we  obtain $\mathbb{N}_{1}=\{1,3,4,5\}$, and $\mathbb{N}_{k}=\{3, 4, 5\}$, where  $k = 3, 4, 5,$ and $\mathbb{N}_{2}=\{2, 3, 4, 5\}$. If we have
\begin{center}
$|\mathbb{N}_{1}\diagdown  (\mathbb{N}_{1}\cap \mathbb{N}_{2})|= |\mathbb{N}_{2}\diagdown  (\mathbb{N}_{1}\cap \mathbb{N}_{2})|=1,$
\end{center}
then we obtain,
\begin{center}
$D(e_{1}),D(e_{2})\in < e_{k}| k\in \mathbb{N}_{1}\cap \mathbb{N}_{2} >\cap N$ = $< e_{3}-e_{4}, e_{3}-e_{5} >$.
\end{center}
Meanwhile, if we have $\mathbb{N}_{k} = \{ 3, 4, 5\}$, then  we obtain,
\begin{center}
$D(e_{3})$, $D(e_{4})$, $D(e_{5})\in <e_{k}| k\in \mathbb{N}_{3}>\cap N$ = $<e_{3}-e_{4}, e_{3}-e_{5}>$.
\end{center}
We find from G.2 that $[D(e_{1})]_{k}=[D(e_{2})]_{k}$, where k = 3, 4, 5. In this case, $[D(e_{1})]=[D(e_{2})]$.\\
\item[Case 2:]If $\gamma(\mathbf{A})=2$, then
 $\mathbf{A}$ an be graphed as follows up to equivalence and isomorphism:
$$
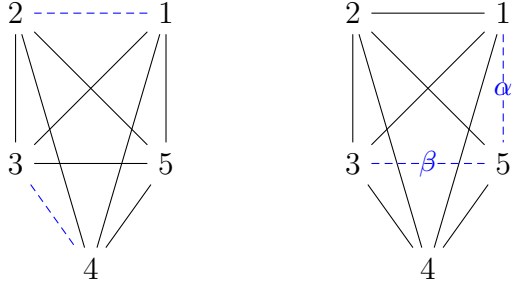
\begin{tikzpicture}\label{35}
[->,>=stealth',shorten >=-5pt,auto,thick,every node/.style={circle,fill=blue!20,draw}]
  \node (n5) at (8,6) {5};
  \node (n2) at (6,8)  {2};
  \node (n4) at (7,4.6)  {4};
  \node (n1) at (8,8) {1};
  \node (n3) at (6,6)  {3};
\path[every node/.style={font=\sffamily\small}]
 (n1) edge node [above] {} (n3)
 (n1) edge node [above] {} (n4)
 (n1) edge node [above] {} (n5)
 (n2) edge node [above] {} (n3)
 (n2) edge node [above] {} (n4)
 (n2) edge node [above] {} (n5)
 (n3) edge node [above] {} (n5)
 (n4) edge node [above] {} (n5);
 \draw[densely dashed,blue] (n1) -- (n2);
  \draw[densely dashed,blue] (n3) -- (n4);
\end{tikzpicture}
\hspace{2cm}
\begin{tikzpicture}\label{36}
[->,>=stealth',shorten >=-5pt,auto,thick,every node/.style={circle,fill=blue!20,draw}]
  \node (n5) at (8,6) {5};
  \node (n2) at (6,8)  {2};
  \node (n4) at (7,4.6)  {4};
  \node (n1) at (8,8) {1};
  \node (n3) at (6,6)  {3};
\path[every node/.style={font=\sffamily\small}]
 (n1) edge node [above] {} (n2)
 (n1) edge node [above] {} (n3)
 (n1) edge node [above] {} (n4)
 (n2) edge node [above] {} (n3)
 (n2) edge node [above] {} (n4)
 (n2) edge node [above] {} (n5)
 (n3) edge node [above] {} (n4)
 (n4) edge node [above] {} (n5);
 \tikzstyle{every node}=[font=\small\itshape]
 \draw[densely dashed,blue]  (n1) -- (n5)node [midway]{$\alpha$};
 \draw[densely dashed,blue]  (n3) -- (n5)node [midway]{$\beta$};
\end{tikzpicture}
$$
where in the first graph, $\mathbb{S}_{1}=\{1 ,3 ,4 ,5\}$,  $\mathbb{S}_{2}=\{2, 3, 4, 5\}$, $\mathbb{S}_{3}=\{1, 2, 3, 5\}$, $\mathbb{S}_{5}=\mathbb{S}$, and  $\mathbb{S}_{4}=\{1, 2, 4, 5\}$, where k = 3, 4, 5. From these, we obtain $\mathbb{N}_{1}=\{1, 5\}$, $\mathbb{N}_{2}=\{2, 5\}$, $\mathbb{N}_{3}=\{3, 5\}$, $\mathbb{N}_{5}=\{5\}$, and  $\mathbb{N}_{5}=\{4, 5\}$. In this case, $D(e_{5})=0$ given that
\begin{center}
$|\mathbb{N}_{k}\diagdown  (\mathbb{N}_{1}\cap \mathbb{N}_{k})|=|\mathbb{N}_{1}\diagdown  (\mathbb{N}_{1}\cap \mathbb{N}_{k})| = |\mathbb{N}_{k}\cap \mathbb{N}_{1}|=1$; . Therefore, $D(e_{k})=0,$ where k = 2, 3, 4, and
\end{center}
$Der(\mathbf{A})$ = 0.\\
As for the second case, we have $\mathbb{S}_{1}=\{1,2,3,4\}$,  $\mathbb{S}_{2}=\mathbb{S}$,  $\mathbb{S}_{3}=\{1, 2, 3, 4\}$,  $\mathbb{S}_{4}=\mathbb{S}$, and $\mathbb{S}_{5}=\{2, 4, 5\}$, and we obtain $\mathbb{N}_{K}=\{1, 2, 3, 4\}$, where k = 1, 3. Given that $\mathbb{N}_{r}=\{2,4\}$ for r = 2, 4 ,$\mathbb{N}_{5}=\{2,4,5\}$; we obtain $D(e_{2})$, $D(e_{4})\in <e_{2}-e_{4}>$.   Following G.2, we have:
\begin{center}
  $[D(e_{1})]_{4}=[D(e_{5})]_{4}\hspace{1.5cm} [D(e_{1})]_{2}=[D(e_{5})]_{2} \Rightarrow D(e_{5})=D(e_{1}).$
\end{center}
Similarly, we can obtain $D(e_{3})=D(e_{5})$. Therefore, $D(e_{3}) = D(e_{5}) = D(e_{1})$. , and $Der(\mathbf{A}) = 0 $ when $\alpha\neq\beta$,  $Der(\mathbf{A})=0$. Following  G.1, if $\alpha=\beta$, then the $f\in L(\mathbf{A})$ set is a deviation if and only if
\begin{center}
$f(e_{1}),f(e_{3}), f(e_{5})\in <e_{i} | i\in\mathbb{N}_{1}\cap\mathbb{N}_{3}\cap\mathbb{N}_{5}> \cap N = < e_{2}-e_{4} >.$\\
\end{center}
\item[Case 3: ] If $\gamma(\mathbf{A})=3$, then
 $\mathbf{A}$ can be graphed up to equivalence and isomorphism as follows:
$$
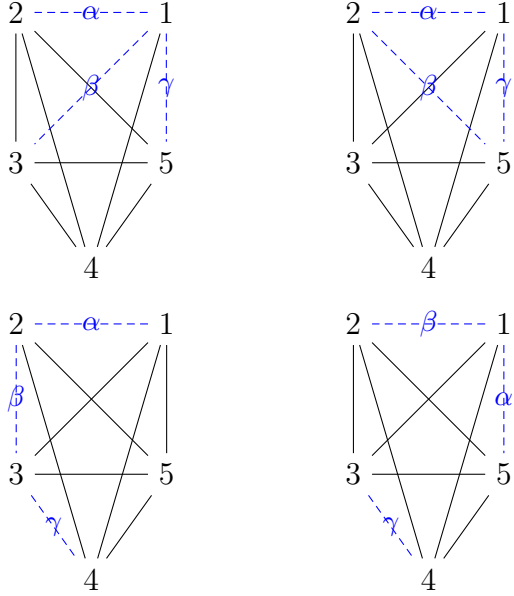
\begin{tikzpicture}\label{35}
[->,>=stealth',shorten >=-5pt,auto,thick,every node/.style={circle,fill=blue!20,draw}]
  \node (n5) at (8,6) {5};
  \node (n2) at (6,8)  {2};
  \node (n4) at (7,4.6)  {4};
  \node (n1) at (8,8) {1};
  \node (n3) at (6,6)  {3};
\path[every node/.style={font=\sffamily\small}]
 (n1) edge node [above] {} (n4)
 (n2) edge node [above] {} (n3)
 (n2) edge node [above] {} (n4)
 (n2) edge node [above] {} (n5)
 (n3) edge node [above] {} (n4)
 (n3) edge node [above] {} (n5)
 (n4) edge node [above] {} (n5);
 \tikzstyle{every node}=[font=\small\itshape]
 \draw[densely dashed,blue]  (n1) -- (n2)node [midway]{$\alpha$};
 \draw[densely dashed,blue]  (n1) -- (n3)node [midway]{$\beta$};
 \draw[densely dashed,blue]  (n1) -- (n5)node [midway]{$\gamma$};
\end{tikzpicture}
\hspace{2cm}
\begin{tikzpicture}\label{35}
[->,>=stealth',shorten >=-5pt,auto,thick,every node/.style={circle,fill=blue!20,draw}]
   \node (n5) at (8,6) {5};
  \node (n2) at (6,8)  {2};
  \node (n4) at (7,4.6)  {4};
  \node (n1) at (8,8) {1};
  \node (n3) at (6,6)  {3};
\path[every node/.style={font=\sffamily\small}]
 (n1) edge node [above] {} (n3)
 (n1) edge node [above] {} (n4)
 (n2) edge node [above] {} (n3)
 (n2) edge node [above] {} (n4)
 (n3) edge node [above] {} (n4)
 (n3) edge node [above] {} (n5)
 (n4) edge node [above] {} (n5);
 \tikzstyle{every node}=[font=\small\itshape]
 \draw[densely dashed,blue]  (n1) -- (n2)node [midway]{$\alpha$};
 \draw[densely dashed,blue]  (n1) -- (n5)node [midway]{$\gamma$};
 \draw[densely dashed,blue]  (n2) -- (n5)node [midway]{$\beta$};
\end{tikzpicture}
$$
$$
\begin{tikzpicture}\label{35}
[->,>=stealth',shorten >=-5pt,auto,thick,every node/.style={circle,fill=blue!20,draw}]
   \node (n5) at (8,6) {5};
  \node (n2) at (6,8)  {2};
  \node (n4) at (7,4.6)  {4};
  \node (n1) at (8,8) {1};
  \node (n3) at (6,6)  {3};
\path[every node/.style={font=\sffamily\small}]
 (n1) edge node [above] {} (n3)
 (n1) edge node [above] {} (n4)
 (n1) edge node [above] {} (n5)
 (n2) edge node [above] {} (n4)
 (n2) edge node [above] {} (n5)
 (n3) edge node [above] {} (n5)
 (n4) edge node [above] {} (n5);
\tikzstyle{every node}=[font=\small\itshape]
 \draw [densely dashed,blue](n1) -- (n2) node [midway]{$\alpha$}   ;
 \draw [densely dashed,blue](n2) -- (n3) node [midway]{$\beta$}   ;
 \draw [densely dashed,blue](n3) -- (n4) node [midway]{$\gamma$} ;
\end{tikzpicture}
\hspace{2cm}
\begin{tikzpicture}\label{35}
[->,>=stealth',shorten >=-5pt,auto,thick,every node/.style={circle,fill=blue!20,draw}]
   \node (n5) at (8,6) {5};
  \node (n2) at (6,8)  {2};
  \node (n4) at (7,4.6)  {4};
  \node (n1) at (8,8) {1};
  \node (n3) at (6,6)  {3};
\path[every node/.style={font=\sffamily\small}]
 (n1) edge node [above] {} (n3)
 (n1) edge node [above] {} (n4)
 (n2) edge node [above] {} (n3)
 (n2) edge node [above] {} (n4)
 (n2) edge node [above] {} (n5)
 (n3) edge node [above] {} (n5)
 (n4) edge node [above] {} (n5);
\tikzstyle{every node}=[font=\small\itshape]
 \draw [densely dashed,blue](n1) -- (n2) node [midway]{$\beta$}   ;
 \draw [densely dashed,blue](n1) -- (n5) node [midway]{$\alpha$}   ;
 \draw [densely dashed,blue](n3) -- (n4) node [midway]{$\gamma$} ;

\end{tikzpicture}
$$
\\In the first case, $\mathbb{S}_{1}=\{1,4\}$, $\mathbb{S}_{2}=\{2,3,4,5\}=\mathbb{S}_{5}$, $\mathbb{S}_{3}=\{2,3,4,5\}$, and $\mathbb{S}_{4}=\mathbb{S}$, from which we obtain $\mathbb{N}_{1}=\{1,4\}$, and $\mathbb{N}_{k}=\{2,3,4,5\}$, where k = 2, 3, 5, and  $\mathbb{N}_{4}=\{4\}$. Therefore, $D(e_{4})=0$. Given that
\begin{center}
$|\mathbb{N}_{1}\cap \mathbb{N}_{k}|=|\mathbb{N}_{1} \diagdown (\mathbb{N}_{1}\cap \mathbb{N}_{k})|=1$ for k = 2, 3, 5, we have
\end{center}

 $D(e_{1})=0$. According to G.1, if $\alpha\neq \beta \neq \gamma$  $\alpha = \beta \neq\gamma$ then we have $Der(\mathbf{A})=0$. In this case, $Der(\mathbf{A})$ is treated as the set of all $f \in L(\mathbf{A})$ where $e_{5}, e_{1}, e_{4}\in ker(f)$ and
\begin{center}
 $f(e_{2}), f(e_{3})\in <e_{2}-e_{3} >.$
\end{center}
  Both $\alpha \neq\beta=\gamma$ and $\alpha =\gamma\neq\beta$  are analogous cases. Specifically, if  $\alpha = \beta=\gamma,$ then  $Der(\mathbf{A})$ can be easily proven as the set of all $\emph{f}\in L(\mathbf{A})$, such as
$e_{1},e_{4}\in ker(\emph{f}$),  that is,
\begin{center}
 $Im (\emph{f})\subseteq <e_{2},e_{3},e_{5}>\cap N=<e_{2}-e_{3}, e_{2}-e_{5} >.$
\end{center}
In the second case, we have $\mathbb{S}_{1}=\{$ 1, 3, 4 $\}$, $\mathbb{S}_{2}=\{$2, 3, 4$\}$, $\mathbb{S}_{3}=\mathbb{S}$, $\mathbb{S}_{4}=\mathbb{S}$, and $\mathbb{S}_{5}=\{$3, 4, 5$\}$, such that $\mathbb{N}_{1}=\{$1, 3, 4$\}$, $\mathbb{N}_{2}=\{$2, 3, 4$\}$, and $\mathbb{N}_{k}=\{$3, 4$\}$ for
 k = 3, 4. $\mathbb{N}_{5}=\{5, 3, 4\}$. From $\mathbb{N}_{r}=\{3, 4\}$,where r =\{ 3, 4\}, we obtain $D(e_{3}), D(e_{4})\in <e_{3}-e_{4}>$. Given that
 \begin{center}
$|\mathbb{N}_{1}\diagdown  (\mathbb{N}_{1}\cap \mathbb{N}_{k})|= |\mathbb{N}_{k}\diagdown  (\mathbb{N}_{1}\cap \mathbb{N}_{k})|=1,$
\end{center}
 where k = 2, 5, we have
\begin{center}
  $D(e_{1}),D(e_{2}),D(e_{5})\in <e_{k}:k\in \mathbb{N}_{1}\cap \mathbb{N}_{k}>\cap N=<e_{3}-e_{4}>$
\end{center}
Following G.2 we have
\begin{center}
  $[D(e_{1})]_{4}=[D(e_{2})]_{4}=[D(e_{5})]_{4}$ and $[D(e_{1})]_{3}=[D(e_{2})]_{3}=[D(e_{5})]_{3}$.
\end{center}
Therefore,
\begin{center}
 $D(e_{1})=D(e_{2})=D(e_{5})$.
\end{center}

In the third case, we have  $\mathbb{S}_{1}=\{$1, 3, 4, 5 $\}$, $\mathbb{S}_{2}=\{$ 2, 4, 5 $\}$, $\mathbb{S}_{3}=\{$ 1, 3, 5 $\}$, $\mathbb{S}_{4}=\{$ 1, 2, 4, 5 $\}$, and $\mathbb{S}_{5}=\mathbb{S}$, and we obtain $\mathbb{N}_{1}=\{$ 1, 5$\}$, $\mathbb{N}_{2}=\{2,4,5\}$, $\mathbb{N}_{3}=\{1,3,5\}$, $\mathbb{N}_{4}=\{4,5\}$, and  $\mathbb{N}_{5}=\{5\}$. Therefore, $D(e_{5})=0$. Given that
\begin{center}
 $|\mathbb{N}_{1} \diagdown (\mathbb{N}_{1}\cap \mathbb{N}_{2})|=|\mathbb{N}_{1} \cap \mathbb{N}_{2}|=1,$
\end{center}
 we have $D(e_{k})=0$, and k = 1. Similarly, on the basis of $e_{4}$ and $e_{3}$, we obtain $D(e_{4})=0$. Following G.2 we  obtain $[D(e_{2})]_{5}=[D(e_{3})]_{5}$, which leads to  $D(e_{3})=D(e_{2})$. We also obtain $[D(e_{2})]_{4}=[D(e_{1})]_{4}$, where, $D(e_{2}) = D(e_{1})=0$, and $D(e_{3})=D(e_{4})=0$, where $Der(\mathbf{A})=0$.
\\\\In the fourth case, we have $\mathbb{S}_{1}=\{1,3,4\}$, $\mathbb{S}_{2}=\{2,3,4,5\}$, $\mathbb{S}_{3}=\{1,2,4,5\}$, $\mathbb{S}_{4}=\{1,2,4,5\}$, and $\mathbb{S}_{5}=\{2,3,4,5\}$, and we  obtain $\mathbb{N}_{1}=\{1\}$ and $\mathbb{N}_{k}=\{2,5\}$, where k = 2, 5,  $\mathbb{N}_{3}=\{3\}$, and $\mathbb{N}_{4}=\{4\}$. Therefore, $D(e_{r})=0$, where r = 1, 3, 4.  According to G.1, However, if $\alpha\neq \beta$, then $Der(\mathbf{A})=0$. If $\alpha = \beta$ then $Der(\mathbf{A})$ represents  the set of all  $\emph{f}\in L(\mathbf{A})$, where $e_{r}\in ker(f)$  for r = 1, 3, 4, such that
\begin{center}
$\emph{f}(e_{2}), \emph{f}(e_{5})\in <e_{2}, e_{5}>\cap N=< e_{2} - e_{5} >$.\\
\end{center}
\item[Case 4:] If $\gamma(\mathbf{A})=4$,
then  $\mathbf{A}$ can be graphed up to equivalence and isomorphism as follows:
$$
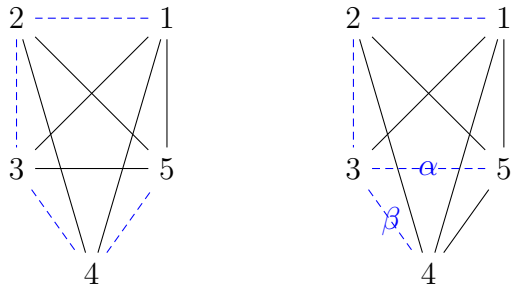
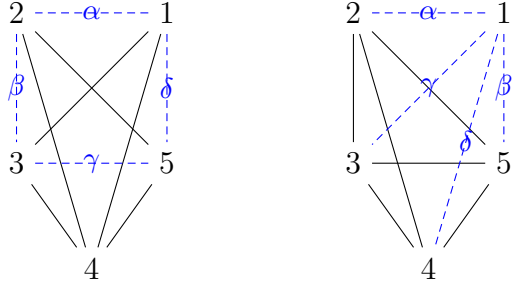
\begin{tikzpicture}\label{35}
[->,>=stealth',shorten >=-5pt,auto,thick,every node/.style={circle,fill=blue!20,draw}]
  \node (n5) at (8,6) {5};
  \node (n2) at (6,8)  {2};
  \node (n4) at (7,4.6)  {4};
  \node (n1) at (8,8) {1};
  \node (n3) at (6,6)  {3};
\path[every node/.style={font=\sffamily\small}]
 (n1) edge node [above] {} (n3)
 (n1) edge node [above] {} (n4)
 (n1) edge node [above] {} (n5)
 (n2) edge node [above] {} (n4)
 (n2) edge node [above] {} (n5)
 (n3) edge node [above] {} (n5);
 \draw[densely dashed,blue] (n1) -- (n2);
 \draw[densely dashed,blue] (n2) -- (n3);
 \draw[densely dashed,blue] (n3) -- (n4);
 \draw[densely dashed,blue] (n4) -- (n5);
\end{tikzpicture}
\hspace{2cm}
\begin{tikzpicture}\label{35}
[->,>=stealth',shorten >=-5pt,auto,thick,every node/.style={circle,fill=blue!20,draw}]
  \node (n5) at (8,6) {5};
  \node (n2) at (6,8)  {2};
  \node (n4) at (7,4.6)  {4};
  \node (n1) at (8,8) {1};
  \node (n3) at (6,6)  {3};
\path[every node/.style={font=\sffamily\small}]
 (n1) edge node [above] {} (n3)
 (n1) edge node [above] {} (n4)
 (n1) edge node [above] {} (n5)
 (n2) edge node [above] {} (n4)
 (n2) edge node [above] {} (n5)
 (n4) edge node [above] {} (n5);
 \draw[densely dashed,blue]  (n1) -- (n2);
 \draw[densely dashed,blue]  (n2) -- (n3);
 \draw[densely dashed,blue]  (n3) -- (n4)node[midway]{$\beta$};
 \draw[densely dashed,blue]  (n3) -- (n5)node[midway]{$\alpha$};
\end{tikzpicture}
$$

$$
\begin{tikzpicture}\label{35}
[->,>=stealth',shorten >=-5pt,auto,thick,every node/.style={circle,fill=blue!20,draw}]
  \node (n5) at (8,6) {5};
  \node (n2) at (6,8)  {2};
  \node (n4) at (7,4.6)  {4};
  \node (n1) at (8,8) {1};
  \node (n3) at (6,6)  {3};
\path[every node/.style={font=\sffamily\small}]
 (n1) edge node [above] {} (n4)
 (n2) edge node [above] {} (n4)
 (n2) edge node [above] {} (n5)
 (n3) edge node [above] {} (n4)
 (n1) edge node [above] {} (n3)
 (n4) edge node [above] {} (n5);
 \tikzstyle{every node}=[font=\small\itshape]
 \draw[densely dashed,blue]  (n1) -- (n2)node [midway]{$\alpha$};
 \draw[densely dashed,blue]  (n2) -- (n3)node [midway]{$\beta$};
 \draw[densely dashed,blue]  (n3) -- (n5)node [midway]{$\gamma$};
 \draw[densely dashed,blue]  (n1) -- (n5)node [midway]{$\delta$};
\end{tikzpicture}
\hspace{2cm}
\begin{tikzpicture}\label{35}
[->,>=stealth',shorten >=-5pt,auto,thick,every node/.style={circle,fill=blue!20,draw}]
  \node (n5) at (8,6) {5};
  \node (n2) at (6,8)  {2};
  \node (n4) at (7,4.6)  {4};
  \node (n1) at (8,8) {1};
  \node (n3) at (6,6)  {3};
\path[every node/.style={font=\sffamily\small}]
 (n2) edge node [above] {} (n3)
 (n2) edge node [above] {} (n4)
 (n2) edge node [above] {} (n5)
 (n3) edge node [above] {} (n4)
 (n3) edge node [above] {} (n5)
 (n4) edge node [above] {} (n5);
 \tikzstyle{every node}=[font=\small\itshape]
 \draw[densely dashed,blue]  (n1) -- (n2)node [midway]{$\alpha$};
 \draw[densely dashed,blue]  (n1) -- (n5)node [midway]{$\beta$};
 \draw[densely dashed,blue]  (n1) -- (n3)node [midway]{$\gamma$};
 \draw[densely dashed,blue]  (n1) -- (n4)node [midway]{$\delta$};

\end{tikzpicture}
$$
In the first case, we have $\mathbb{S}_{1}=\{1,3,4,5\}$, $\mathbb{S}_{2}=\{2,4,5\}$, $\mathbb{S}_{3}=\{1,3,5\}$, $\mathbb{S}_{4}=\{1,2,4\}$, and $\mathbb{S}_{5}=\{1,2,3,5\}$, and we obtain $\mathbb{N}_{k}=\{k\}$, for k = 1, 2, 4, 5 and  $\mathbb{N}_{3}=\{1,3,5\}$. Therefore, $D(e_{k})$ = 0, k=1, 2, 4, 5. Following G.2, we have $[D(e_{3})]_{5}=[D(e_{2})]_{5}$, where $D(e_{3})=D(e_{2})=0$, and $Der(\mathbf{A})=0$.
\\\\In the second case, we have $\mathbb{S}_{1}=\{1,3,4,5\}$, $\mathbb{S}_{2}=\{2,4,5\}$, $\mathbb{S}_{3}=\{1,3\}$, $\mathbb{S}_{4}=\{1,2,4,5\}$, and $\mathbb{S}_{5}=\{1,2,4,5\}$, and we obtain $\mathbb{N}_{1}=\{1\}$, $\mathbb{N}_{2}=\{2,4,5\}$, $\mathbb{N}_{3}=\{1,3\}$, $\mathbb{N}_{4}=\{4,5\}$, and $\mathbb{N}_{5}=\{4,5\}$, where $D(e_{1})=0$. Following G.2 , we have $[D(e_{1})]_{s}=[D(e_{2})]_{s}=0$ (where s = 4, 5).   $[D(e_{3})]_{1}=[D(e_{4})]_{1}$, and  $[D(e_{3})]=[D(e_{4})]$ , and  $[D(e_{3})]=[D(e_{5})]$. If $\alpha \neq \beta$,  then we have $Der(\mathbf{A})=0$, but if $\alpha \neq \beta $, then we have $e_{2},e_{1}\in ker(\mathbf{A})$. Therefore, $f(e_{4}), f(e_{5}) \in N \cap \mathbb{N}_{4}=<e_{4}-e_{5}>$.
 \\\\In the third case, we have $\mathbb{S}_{1}=\{1,3,4\}$, $\mathbb{S}_{2}=\{2,4,5\}$, $\mathbb{S}_{3}=\{1,3,4\}$, $\mathbb{S}_{4}=\mathbb{S}$, and $\mathbb{S}_{5}=\{2,4,5\}$, and $\mathbb{N}_{r}=\{1, 3, 4\}$, (where r=1, 3) $\mathbb{N}_{r}=\{2, 4, 5\}$ (for r= 2, 5). $\mathbb{N}_{4}=\{4\}$.  We obtain $D(e_{4})=0$. Following G.1, if either $\alpha \neq  \beta$ or $\gamma \neq \delta,$ then $D(e_{1})=D(e_{3})=0$,  $\alpha \neq  \delta$ or $\gamma\neq\beta$ , and

   $D(e_{2}),D(e_{5})=0$
. We rearrange the induces  if  necessary and  assume without loss of generality that  if $\alpha =  \beta \neq \gamma = \delta$, then $\emph{f}\in \emph{L}(\mathbf{A})$ is a derivation of $\mathbf{A}$ if and only if $e_{4}\in ker(f)$ and  $\emph{f}(e_{1})$, $\emph{f}(e_{3})\in<e_{1}-e_{3}>$, and $\emph{f}(e_{5})$,$\emph{f}(e_{2})\in<e_{2}-e_{5}>$.
\\\\ In the fourth case, we have $\mathbb{S}_{1}=\{1\}$, and $\mathbb{S}_{k}=\{ 2, 3, 4, 5\}$, where k = 2, 3, 4, 5. We subsequently obtain $\mathbb{N}_{1}=1$, $D(e_{1})=0$, and $\mathbb{N}_{k}=\{2,3,4,5\}$, and k = 2, 3, 4, 5. The following cases are derived accordingly:\\
1)By letting $\alpha= \beta =\gamma=\delta$, $Der(\mathbf{A})$ can be easily proven to be a set of all $\emph{f}\in \emph{L}\mathbf{(A)}$ such that $e_{1}\in ker(f)$ and
\begin{center}
$Im(\emph{f})\subset <e_{2},e_{3},e_{4},e_{5}>\cap N=<e_{2}-e_{3},e_{2}-e_{4},e_{2}-e_{5}>.$
\end{center}

2)By letting $\alpha= \beta =\gamma\neq\delta$, we acquire G.1 forces $e_{4},e_{1}\in ker(f)$  and
\begin{center}
$D(e_{2}),D(e_{3}),D(e_{5})\in U\cap N= <e_{2}-e_{3},e_{2}-e_{5}>.$
\end{center}
Cases $\alpha\neq \beta =\gamma=\delta $, $\alpha=\gamma=\delta\neq \beta $, and $\alpha=\gamma=\beta\neq \delta$ are all analogous.\\

3) By letting $\alpha= \beta \neq\gamma=\delta$, $Der(\mathbf{A})$ can be easily proven to be a set of all $\emph{f}\in \emph{L}(\mathbf{A})$ such that  $f(e_{5})$, $\emph{f}(e_{2})\in <e_{2}-e_{5}>$,  $\emph{f}(e_{3})$, $\emph{f}(e_{4})\in < e_{3} - e_{4} >$  and $e_{1}\in ker(f)$. Cases
$\alpha= \beta \neq\gamma=\delta$, $\alpha= \gamma \neq\beta=\delta$, $\alpha= \delta \neq\gamma=\beta$ are all analogous.\\
4) By letting $\alpha= \beta \neq\gamma\neq\delta$, we have $ e_{1}, e_{4}, e_{3}\in ker(\emph{f})$, $\emph{f}(e_{2}),\emph{f}(e_{5})\in <e_{2}-e_{5}>$. Cases $\alpha\neq\beta \neq\gamma=\delta$ is analogous.\\
\item[Case 5: ]If $\gamma(\mathbf{A})=5$, $\mathbf{A}$ can be graphed up to equivalence and isomorphism  as follows:
$$
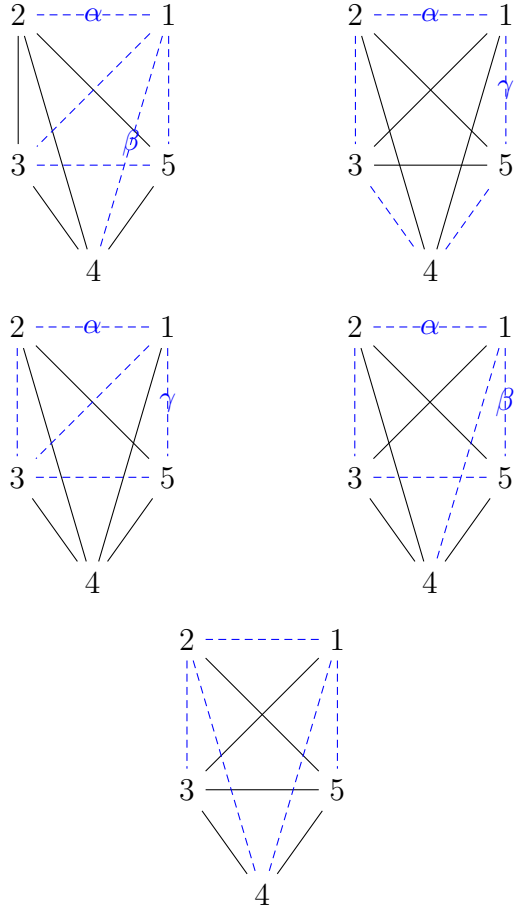
\begin{tikzpicture}\label{1}
[->,>=stealth',shorten >=-5pt,auto,thick,every node/.style={circle,fill=blue!10,draw}]
  \node (n5) at (8,6) {5};
  \node (n2) at (6,8)  {2};
  \node (n4) at (7,4.6)  {4};
  \node (n1) at (8,8) {1};
  \node (n3) at (6,6)  {3};
\path[every node/.style={font=\sffamily\small}]
 (n2) edge node [above] {} (n3)
 (n2) edge node [above] {} (n4)
 (n2) edge node [above] {} (n5)
 (n3) edge node [above] {} (n4)
 (n4) edge node [above] {} (n5);
 \tikzstyle{every node}=[font=\small\itshape]
 \draw[densely dashed,blue]  (n1) -- (n2)node [midway]{$\alpha$};
 \draw[densely dashed,blue]  (n1) -- (n3);
 \draw[densely dashed,blue]  (n1) -- (n4)node [midway]{$\beta$};
 \draw[densely dashed,blue]  (n1) -- (n5);
 \draw[densely dashed,blue]  (n3) -- (n5);
\end{tikzpicture}
\hspace{2cm}
\begin{tikzpicture}\label{1}
[->,>=stealth',shorten >=-5pt,auto,thick,every node/.style={circle,fill=blue!10,draw}]
  \node (n5) at (8,6) {5};
  \node (n2) at (6,8)  {2};
  \node (n4) at (7,4.6)  {4};
  \node (n1) at (8,8) {1};
  \node (n3) at (6,6)  {3};
\path[every node/.style={font=\sffamily\small}]
 (n1) edge node [above] {} (n3)
 (n1) edge node [above] {} (n4)
 (n2) edge node [above] {} (n4)
 (n2) edge node [above] {} (n5)
 (n3) edge node [above] {} (n5);
 \tikzstyle{every node}=[font=\small\itshape]
 \draw[densely dashed,blue] (n1) -- (n2)node [midway]{$\alpha$};
 \draw[densely dashed,blue] (n2) -- (n3);
 \draw[densely dashed,blue] (n3) -- (n4);
 \draw[densely dashed,blue] (n4) -- (n5);
 \draw[densely dashed,blue] (n1) -- (n5)node [midway]{$\gamma$};;
\end{tikzpicture}
$$
$$
\begin{tikzpicture}\label{1}
[->,>=stealth',shorten >=-5pt,auto,thick,every node/.style={circle,fill=blue!10,draw}]
 \node (n5) at (8,6) {5};
  \node (n2) at (6,8)  {2};
  \node (n4) at (7,4.6)  {4};
  \node (n1) at (8,8) {1};
  \node (n3) at (6,6)  {3};
\path[every node/.style={font=\sffamily\small}]
 (n1) edge node [above] {} (n4)
 (n2) edge node [above] {} (n4)
 (n2) edge node [above] {} (n5)
 (n3) edge node [above] {} (n4)
 (n4) edge node [above] {} (n5);
 \tikzstyle{every node}=[font=\small\itshape]
 \draw[densely dashed,blue] (n1) -- (n2)node [midway]{$\alpha$};
 \draw[densely dashed,blue] (n2) -- (n3);
 \draw[densely dashed,blue] (n3) -- (n5);
 \draw[densely dashed,blue] (n1) -- (n5)node [midway]{$\gamma$};
 \draw[densely dashed,blue] (n1) -- (n3);
\end{tikzpicture}
\hspace{2cm}
\begin{tikzpicture}\label{1}
[->,>=stealth',shorten >=-5pt,auto,thick,every node/.style={circle,fill=blue!10,draw}]
 \node (n5) at (8,6) {5};
  \node (n2) at (6,8)  {2};
  \node (n4) at (7,4.6)  {4};
  \node (n1) at (8,8) {1};
  \node (n3) at (6,6)  {3};
\path[every node/.style={font=\sffamily\small}]
 (n1) edge node [above] {} (n3)
 (n2) edge node [above] {} (n4)
 (n2) edge node [above] {} (n5)
 (n3) edge node [above] {} (n4)
 (n4) edge node [above] {} (n5);
 \draw[densely dashed,blue] (n1) -- (n2)node [midway]{$\alpha$};
 \draw[densely dashed,blue] (n2) -- (n3);
 \draw[densely dashed,blue] (n3) -- (n5);
 \draw[densely dashed,blue] (n1) -- (n5)node [midway]{$\beta$};
 \draw[densely dashed,blue] (n1) -- (n4);
\end{tikzpicture}
$$
$$
\begin{tikzpicture}\label{1}
[->,>=stealth',shorten >=-5pt,auto,thick,every node/.style={circle,fill=blue!10,draw}]
  \node (n5) at (8,6) {5};
  \node (n2) at (6,8)  {2};
  \node (n4) at (7,4.6)  {4};
  \node (n1) at (8,8) {1};
  \node (n3) at (6,6)  {3};
\path[every node/.style={font=\sffamily\small}]
 (n1) edge node [above] {} (n3)
 (n2) edge node [above] {} (n5)
 (n3) edge node [above] {} (n4)
 (n3) edge node [above] {} (n5)
  (n4) edge node [above] {} (n5);
 \draw[densely dashed,blue] (n1) -- (n2);
 \draw[densely dashed,blue] (n2) -- (n3);
 \draw[densely dashed,blue] (n1) -- (n4);
 \draw[densely dashed,blue] (n2) -- (n4);
 \draw[densely dashed,blue] (n1) -- (n5);
\end{tikzpicture}
$$
In the first case, we have  $\mathbb{S}_{1}=\{1\}$, $\mathbb{S}_{2}=\{2,3,4,5\}$, $\mathbb{S}_{3}=\{2,3,4\}$, $\mathbb{S}_{4}=\{2,3,4,5\}$, and $\mathbb{S}_{5}=\{2,4,5\}$, and we obtain $\mathbb{N}_{1}=\{1\}$, $\mathbb{N}_{2}=\{2,4\}$, $\mathbb{N}_{3}=\{2,3,4\}$, $\mathbb{N}_{4}=\{2,4\}$, and $\mathbb{N}_{5}=\{2,4,5\}$. Therefore, $D(e_{1})=0$ and,
\begin{center}
$|\mathbb{N}_{3} \diagdown (\mathbb{N}_{5}\cap \mathbb{N}_{3})|=|\mathbb{N}_{5} \diagdown (\mathbb{N}_{3}\cap \mathbb{N}_{5})|=1.$
\end{center}
In this case, $\emph{D}(e_{3}),\emph{D}(e_{5})\in U \cap
N=<e_{2}-e_{4}>$. Following G.1, if $\alpha\neq \beta$, then $e_{2},e_{4}\in ker(\emph{f})$.whereas if $\alpha = \beta$, then $\emph{f}\in L(\mathbf{A})$ is considered a derivation of $\mathbf{A}$ if and only if $e_{1}\in ker(f)$ and  $\emph{f}(e_{2}),\emph{f}(e_{4})\in U \cap N =<e_{2}-e_{4}>$ .
\\\\In the second case, we have $\mathbb{S}_{1}=\{1,3,4\}$, $\mathbb{S}_{2}=\{2,4,5\}$, $\mathbb{S}_{3}=\{1,3,5\}$, $\mathbb{S}_{4}=\{1,2,4\}$, and $\mathbb{S}_{5}=\{2,3,5\}$. and, we obtain $\mathbb{N}_{k}=\{k\}$, k=1, 2, 3, 4, 5. Therefore, $Der(\mathbf{A})$ = 0.
\\\\ In the third case, we have $\mathbb{S}_{1}=\{1,4\}$, $\mathbb{S}_{2}=\mathbb{S}_{5}=\{2,4,5\}$, $\mathbb{S}_{3}=\{3,4\}$, $\mathbb{S}_{4}=\{2, 4, 5\}$. and we obtain $\mathbb{N}_{1}=\{1,4\}$, $\mathbb{N}_{k}=\{2,4,5\}$, (where k = 2, 5) $\mathbb{N}_{3}=\{3,4\}$, and $\mathbb{N}_{4}=\{4\}$. In this case, $D(e_{4})=0$. Given that
\begin{center}
$|\mathbb{N}_{1}\cap \mathbb{N}_{3}|=|\mathbb{N}_{3}\setminus (\mathbb{N}_{1}\cap \mathbb{N}_{3})|=
|\mathbb{N}_{1}\setminus (\mathbb{N}_{3}\cap \mathbb{N}_{1})|=1,$
\end{center}
 we have $D(e_{1}),D(e_{3})=0$. If $\alpha\neq \gamma$,  then $Der(\mathbf{A})$=0, whereas if $\alpha=\gamma$, then the set of all $f\in L(\mathbf{A})$, similar that $e_{1}, e_{3}, e_{4}\in ker(f)$ and
\begin{center}
$\emph{f}(e_{2}),\emph{f}(e_{5})\in <e_{2},e_{5}> \cap N=<e_{2}-e_{5}>$
\end{center}
.
\\\\In the fourth case, we have $\mathbb{S}_{1}=\{1,3\}$, $\mathbb{S}_{2}=\{2,4,5\}$, $\mathbb{S}_{3}=\{1,3,4\}$, $\mathbb{S}_{4}=\{2,3,4,5\}$, and $\mathbb{S}_{5}=\{2,4,5\}$ and we obtain $\mathbb{N}_{1}=\{1,3\}$, $\mathbb{N}_{2}=\{2,4,5\}$, $\mathbb{N}_{3}=\{3\}$, $\mathbb{N}_{4}=\{4\}$, and $\mathbb{N}_{5}=\{2,4,5\}$.
Therefore, $D(e_{3}),D(e_{4})=0$. By using the previous theorem, we obtain
  \begin{center}
    $|\mathbb{N}_{1} \diagdown (\mathbb{N}_{3}\cap \mathbb{N}_{1})|=|\mathbb{N}_{3}\cap \mathbb{N}_{1}|=1$.
  \end{center}
  and show that, $D(e_{1})=0$. Following G.1, if $\alpha\neq\beta$, then $Der(\mathbf{A})=0$, whereas if $\alpha = \beta$, then the set of all $f\in L(\mathbf{A})$ similar $e_{1},e_{3},e_{4}\in ker(f)$ and
\begin{center}
 $\emph{f}(e_{2}),\emph{f}(e_{5})\in <e_{2},e_{5}> \cap N=<e_{2}-e_{5}>$
 \end{center}
In the fifth case, we have $\mathbb{S}_{1}=\{ 1, 3\}$, $\mathbb{S}_{2}=\{ 2, 5\}$, $\mathbb{S}_{3}=\{ 1, 3, 4, 5\}$, $\mathbb{S}_{4}=\{ 3, 4, 5\}$, and $\mathbb{S}_{5}=\{2, 3, 4, 5\}$ and we obtain $\mathbb{N}_{1}=\{1,3\}$, $\mathbb{N}_{2}=\{2, 5\}$, $\mathbb{N}_{3}=\{3\}$, $\mathbb{N}_{4}=\{3, 4, 5\}$, and $\mathbb{N}_{5}=\{5\}$. Therefore, $D(e_{k})$ = 0 for k = 3, 4, and
 \begin{center}
    $|\mathbb{N}_{2} \diagdown (\mathbb{N}_{2}\cap \mathbb{N}_{4})|=|\mathbb{N}_{2}\cap \mathbb{N}_{4}|=1.$
  \end{center}

In this case, $D(e_{2})=0$. Following G.2, we obtain $[D(e_{2})]_{5}=[D(e_{4})]_{5}=0$ , which suggests that $D(e_{4})=0$. We also obtain  $D(e_{1})=0$, thereby suggesting that $Der(\mathbf{A})$=0.\\

\item[Case 6: ] If $\gamma(\mathbf{A})=6$, then  $\mathbf{A}$ can be graphed up to equivalence and isomorphism  as follows:
$$
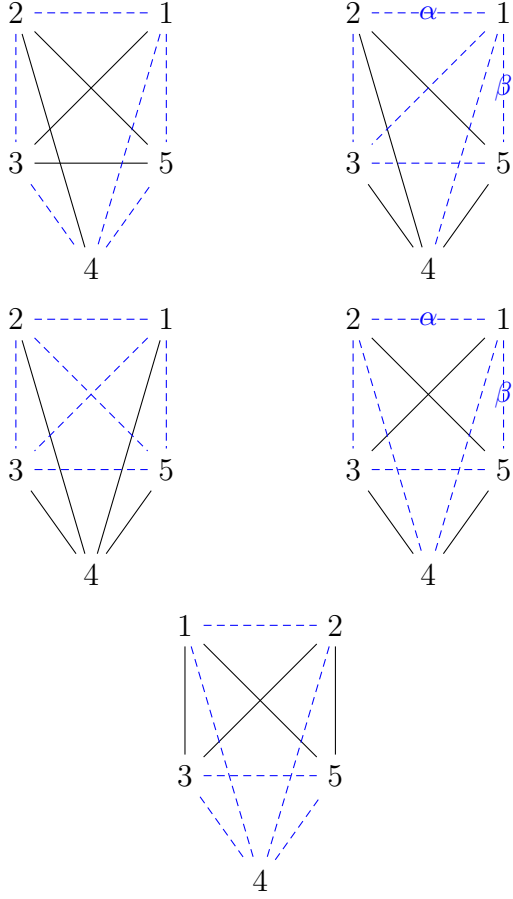
\begin{tikzpicture}\label{sss}
[->,>=stealth',shorten >=-5pt,auto,thick,every node/.style={circle,fill=blue!20,draw}]
  \node (n5) at (8,6) {5};
  \node (n2) at (6,8)  {2};
  \node (n4) at (7,4.6)  {4};
  \node (n1) at (8,8) {1};
  \node (n3) at (6,6)  {3};
\path[every node/.style={font=\sffamily\small}]
 (n1) edge node [above] {} (n3)
 (n2) edge node [above] {} (n4)
 (n2) edge node [above] {} (n5)
 (n3) edge node [above] {} (n5);
 \draw[densely dashed,blue] (n1) -- (n2);
 \draw[densely dashed,blue] (n2) -- (n3);
 \draw[densely dashed,blue] (n3) -- (n4);
 \draw[densely dashed,blue] (n4) -- (n5);
 \draw[densely dashed,blue] (n1) -- (n5);
 \draw[densely dashed,blue] (n1) -- (n4);
\end{tikzpicture}
\hspace{2cm}
\begin{tikzpicture}\label{121}
[->,>=stealth',shorten >=-5pt,auto,thick,every node/.style={circle,fill=blue!20,draw}]
  \node (n5) at (8,6) {5};
  \node (n2) at (6,8)  {2};
  \node (n4) at (7,4.6)  {4};
  \node (n1) at (8,8) {1};
  \node (n3) at (6,6)  {3};
\path[every node/.style={font=\sffamily\small}]
 (n2) edge node [above] {} (n4)
 (n4) edge node [above] {} (n5)
 (n2) edge node [above] {} (n5)
 (n3) edge node [above] {} (n4);
 \tikzstyle{every node}=[font=\small\itshape]
 \draw[densely dashed,blue] (n1) -- (n2)node [midway]{$\alpha$};
 \draw[densely dashed,blue] (n2) -- (n3);
 \draw[densely dashed,blue] (n3) -- (n5);
 \draw[densely dashed,blue] (n1) -- (n5)node [midway]{$\beta$};
 \draw[densely dashed,blue] (n1) -- (n3);
 \draw[densely dashed,blue] (n1) -- (n4);
\end{tikzpicture}
$$
$$
\begin{tikzpicture}\label{121}
[->,>=stealth',shorten >=-5pt,auto,thick,every node/.style={circle,fill=blue!20,draw}]
  \node (n5) at (8,6) {5};
  \node (n2) at (6,8)  {2};
  \node (n4) at (7,4.6)  {4};
  \node (n1) at (8,8) {1};
  \node (n3) at (6,6)  {3};
\path[every node/.style={font=\sffamily\small}]
 (n1) edge node [above] {} (n4)
 (n2) edge node [above] {} (n4)
 (n3) edge node [above] {} (n4)
 (n4) edge node [above] {} (n5);
 \draw[densely dashed,blue] (n1) -- (n2);
 \draw[densely dashed,blue] (n2) -- (n3);
 \draw[densely dashed,blue] (n3) -- (n5);
 \draw[densely dashed,blue] (n1) -- (n5);
 \draw[densely dashed,blue] (n1) -- (n3);
 \draw[densely dashed,blue] (n2) -- (n5);
\end{tikzpicture}
\hspace{2cm}
\begin{tikzpicture}\label{121}
[->,>=stealth',shorten >=-5pt,auto,thick,every node/.style={circle,fill=blue!20,draw}]
   \node (n5) at (8,6) {5};
  \node (n2) at (6,8)  {2};
  \node (n4) at (7,4.6)  {4};
  \node (n1) at (8,8) {1};
  \node (n3) at (6,6)  {3};
\path[every node/.style={font=\sffamily\small}]
 (n2) edge node [above] {} (n5)
 (n3) edge node [above] {} (n4)
 (n1) edge node [above] {} (n3)
 (n4) edge node [above] {} (n5);
 \tikzstyle{every node}=[font=\small\itshape]
 \draw[densely dashed,blue] (n1) -- (n2)node [midway]{$\alpha$};
 \draw[densely dashed,blue] (n2) -- (n3);
 \draw[densely dashed,blue] (n3) -- (n5);
 \draw[densely dashed,blue] (n1) -- (n5)node [midway]{$\beta$};
 \draw[densely dashed,blue] (n1) -- (n4);
 \draw[densely dashed,blue] (n2) -- (n4);
\end{tikzpicture}
$$
$$
\begin{tikzpicture}\label{121}
[->,>=stealth',shorten >=-5pt,auto,thick,every node/.style={circle,fill=blue!20,draw}]
  \node (n5) at (8,6) {5};
  \node (n1) at (6,8)  {1};
  \node (n4) at (7,4.6)  {4};
  \node (n2) at (8,8) {2};
  \node (n3) at (6,6)  {3};
\path[every node/.style={font=\sffamily\small}]
 (n1) edge node [above] {} (n3)
 (n1) edge node [above] {} (n5)
 (n2) edge node [above] {} (n5)
 (n2) edge node [above] {} (n3);
 \draw[densely dashed,blue] (n1) -- (n2);
 \draw[densely dashed,blue] (n1) -- (n4);
 \draw[densely dashed,blue] (n2) -- (n4);
 \draw[densely dashed,blue] (n3) -- (n5);
 \draw[densely dashed,blue] (n3) -- (n4);
 \draw[densely dashed,blue] (n4) -- (n5);
\end{tikzpicture}
$$
In the first case, we have $\mathbb{S}_{1}=\{1,3\}$, $\mathbb{S}_{2}=\{2,4,5\}$, $\mathbb{S}_{3}=\{1,3,5\}$, $\mathbb{S}_{4}=\{2,4\}$, and $\mathbb{S}_{5}=\{2,3,5\}$, and we obtain $\mathbb{N}_{1}=\{1,3\}$, $\mathbb{N}_{2}=\{2\}$, $\mathbb{N}_{3}=\{3\}$, $\mathbb{N}_{4}=\{2,4\}$, and $\mathbb{N}_{5}=\{5\}$. Therefore, $D(e_{2}),D(e_{3}),D(e_{5})=0$. Following G.2,
\begin{center}
$[D(e_{5})]_{2}=[D(e_{4})]_{2}=0$,
\end{center}
which suggests that, $D(e_{4})=0$, $D(e_{1})=0$ and $Der(\mathbf{A})$ = 0.
\\In the second case, we have $\mathbb{S}_{1}=\{1\}$, $\mathbb{S}_{2}=\{2,4,5\}$, $\mathbb{S}_{3}=\{3,4\}$, $\mathbb{S}_{4}=\{2,3,4,5\}$, and $\mathbb{S}_{5}=\{2,4,5\}$ and we obtain $\mathbb{N}_{1}=\{1\}$, $\mathbb{N}_{2}=\{2,4,5\}$, $\mathbb{N}_{3}=\{3,4\}$,
$\mathbb{N}_{4}=\{4\}$, and $\mathbb{N}_{5}=\{2,4,5\}$. Therefore, acquire $D(e_{1}),D(e_{4})=0$. Following the previous theorem,
\begin{center}
$|\mathbb{N}_{3}\diagdown (\mathbb{N}_{2}\cap \mathbb{N}_{3})|=|\mathbb{N}_{2}\cap \mathbb{N}_{3}| =1,$
\end{center}
which means that, $D(e_{3})=0$. With reference to the same theorem, we let $N_{k}=\{ 2, 4, 5\}$ for $k = 2, 5$. In this case, if $\alpha=\beta$, then the  $f \in L(\mathbf{A})$ set can be treated as a derivation if and only if $e_{i}\in ker(f)$, where i = 1, 3, 4 and  $\emph{f}(e_{2}),\emph{f}(e_{5})\in <e_{2}-e_{5}>$, and if $\alpha\neq \beta.$ Therefore, $Der(\mathbf{A})=0$.
\\\\In the third case, we have $\mathbb{S}_{1}=\{1,4\}$, $\mathbb{S}_{2}=\{2,4\}$, $\mathbb{S}_{3}=\{3,4\}$, $\mathbb{S}_{1}=\mathbb{S}$, and $\mathbb{S}_{5}=\{4,5\}$, and we obtain $\mathbb{N}_{1}=\{1,4\}$, $\mathbb{N}_{2}=\{2,4\}$, $\mathbb{N}_{3}=\{3,4\}$, $\mathbb{N}_{4}=\{4\}$, and $\mathbb{N}_{5}=\{4,5\}$. Therefore, $D(e_{4})=0$. Given that
\begin{center}
$|\mathbb{N}_{1}\diagdown (\mathbb{N}_{1}\cap \mathbb{N}_{k})|=|\mathbb{N}_{k}\diagdown (\mathbb{N}_{1}\cap \mathbb{N}_{k})|=|\mathbb{N}_{1}\cap \mathbb{N}_{k}| =1.$
\end{center}
 we obtain $D(e_{k})=0$, where $ k = 1, 2, 3, 5$. In this case, $Der(\mathbf{A})= 0.$
\\\\In the fourth case, we have $\mathbb{S}_{1}=\{1,3\}$, $\mathbb{S}_{2}=\{2,5\}$, $\mathbb{S}_{3}=\{1,3\}$, $\mathbb{S}_{4}=\{3,4,5\}$, and $\mathbb{S}_{5}=\{2,4,5\}$, so that $\mathbb{N}_{1}=\{1,3\}$, $\mathbb{N}_{2}=\mathbb{N}_{5}=\{2,5\}$, $\mathbb{N}_{3}=\{3\}$, and $\mathbb{N}_{4}=\{4\}$.  Therefore, $D(e_{k})=0$ , where k = 3, 4. Following G.2, we obtain
\begin{center}
  $[D(e_{1})]_{3}=[D(e_{4})]_{3} \Rightarrow D(e_{1})=0$.
\end{center}
If $\alpha\neq\beta$, then  $Der(\mathbf{A})=0$, whereas  if $\alpha=\beta,$ then $Der(\mathbf{A})$ represents  the set of all $f\in L(\mathbf{A})$, such as $e_{r}\in ker(f)$, where r = 1, 3, 4. $\emph{f}(e_{2}),\emph{f}(e_{5})\in <e_{2}-e_{5}>$.
\\\\In the fifth case, we have $\mathbb{S}_{1}=\{1,3,5\}$, $\mathbb{S}_{2}=\{2,3,5\}$, $\mathbb{S}_{3}=\{1,2,3\}$, $\mathbb{S}_{4}=\{4\}$, and $\mathbb{S}_{5}=\{1,2,5\}$ and we obtain,
$N_{k}=\{k\}$, where k = 1, 2, 3, 4, 5. In this case, $Der(\mathbf{A})=0$.\\
\item[Case 7: ]If $\gamma(\mathbf{A})=7,$
then  $\mathbf{A}$ can be graphed up to equivalence and isomorphism  as follows:
$$
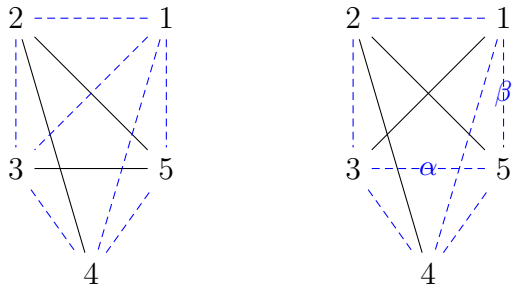
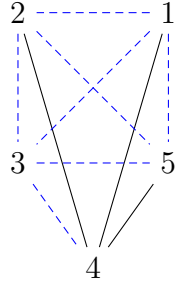
\begin{tikzpicture}\label{121}
[->,>=stealth',shorten >=-5pt,auto,thick,every node/.style={circle,fill=blue!20,draw}]
  \node (n5) at (8,6) {5};
  \node (n2) at (6,8)  {2};
  \node (n4) at (7,4.6)  {4};
  \node (n1) at (8,8) {1};
  \node (n3) at (6,6)  {3};
\path[every node/.style={font=\sffamily\small}]
 (n2) edge node [above] {} (n5)
 (n2) edge node [above] {} (n4)
 (n3) edge node [above] {} (n5);
 \draw[densely dashed,blue] (n1) -- (n2);
 \draw[densely dashed,blue] (n2) -- (n3);
 \draw[densely dashed,blue] (n3) -- (n4);
 \draw[densely dashed,blue] (n4) -- (n5);
 \draw[densely dashed,blue] (n1) -- (n3);
 \draw[densely dashed,blue] (n1) -- (n4);
  \draw[densely dashed,blue] (n1) -- (n5);
\end{tikzpicture}
\hspace{2cm}
\begin{tikzpicture}\label{32}
[->,>=stealth',shorten >=-5pt,auto,thick,every node/.style={circle,fill=blue!20,draw}]
  \node (n5) at (8,6) {5};
  \node (n2) at (6,8)  {2};
  \node (n4) at (7,4.6)  {4};
  \node (n1) at (8,8) {1};
  \node (n3) at (6,6)  {3};
\path[every node/.style={font=\sffamily\small}]
 (n2) edge node [above] {} (n4)
 (n1) edge node [above] {} (n3)
 (n2) edge node [above] {} (n5);
 \tikzstyle{every node}=[font=\small\itshape]
 \draw[densely dashed,blue] (n1) -- (n2);
 \draw[densely dashed,blue] (n2) -- (n3);
 \draw[densely dashed,blue] (n3) -- (n4);
 \draw[densely dashed,blue] (n4) -- (n5);
 \draw[densely dashed,blue] (n1) -- (n4);
 \draw[densely dashed,blue] (n3) -- (n5)node [midway]{$\alpha$};
 \draw[densely dashed,blue] (n1) -- (n5)node [midway]{$\beta$};
\end{tikzpicture}
$$
$$
\begin{tikzpicture}\label{23}
[->,>=stealth',shorten >=-5pt,auto,thick,every node/.style={circle,fill=blue!20,draw}]
  \node (n5) at (8,6) {5};
  \node (n2) at (6,8)  {2};
  \node (n4) at (7,4.6)  {4};
  \node (n1) at (8,8) {1};
  \node (n3) at (6,6)  {3};
\path[every node/.style={font=\sffamily\small}]
 (n1) edge node [above] {} (n4)
 (n2) edge node [above] {} (n4)
 (n4) edge node [above] {} (n5);
 \draw[densely dashed,blue] (n1) -- (n2);
 \draw[densely dashed,blue] (n2) -- (n3);
 \draw[densely dashed,blue] (n3) -- (n4);
 \draw[densely dashed,blue] (n3) -- (n5);
 \draw[densely dashed,blue] (n1) -- (n5);
 \draw[densely dashed,blue] (n1) -- (n3);
 \draw[densely dashed,blue] (n2) -- (n5);
\end{tikzpicture}
$$
In the first case, we have $\mathbb{S}_{1}=\{1 \}$, $\mathbb{S}_{2}=\{2,4,5 \}$, $\mathbb{S}_{3}=\{3,5 \}$, $\mathbb{S}_{4}=\{2,4 \}$, and
$\mathbb{S}_{5}=\{ 2,3,5 \}$ and we obtain, $ \mathbb{N}_{k}=\{ k \}$, for k = 1, 2, $ \mathbb{N}_{3}=\{3,5 \}$, $\mathbb{N}_{4}=\{ 2,4 \}$, and $\mathbb{N}_{5}=\{ 2,5 \}$. Given that
\begin{center}
 $|\mathbb{N}_{4} \diagdown ( \mathbb{N}_{5}\cap \mathbb{N}_{4})|=|\mathbb{N}_{4} \cap \mathbb{N}_{5}|=|
\mathbb{N}_{4} \diagdown (\mathbb{N}_{3}\cap \mathbb{N}_{5})|=1.$
\end{center}
we obtain $D(e_{4}),D(e_{5})=0$.
Following G.2,
\begin{center}
  $[D(e_{2})]_{5}=[D(e_{3})]_{5} \Rightarrow D(e_{3})=0.$
\end{center}
In this case, $Der(\mathbf{A})=0$.
\\\\In the second case, we have $\mathbb{S}_{1}=\{1,3\}$, $\mathbb{S}_{2}=\{2,4,5\}$,
 $\mathbb{S}_{3}=\{1,3\}$, $ \mathbb{S}_{4}=\{2,4\}$, and $\mathbb{S}_{5}=\{2,5\}$ and we obtain
$\mathbb{N}_{1}=\{1,3\}$,$\mathbb{N}_{2}=\{2\}$, $\mathbb{N}_{3}=\{1,3\}$,  $\mathbb{N}_{4}=\{4\}$, and
$\mathbb{N}_{5}=\{2,5\}$. Therefore, $D(e_{4}), D(e_{2}) = 0$. Given that
\begin{center}
  $[D(e_{5})]_{2}=[D(e_{4})]_{2} \Rightarrow D(e_{5})=0.$
\end{center}
we have $D(e_{5})=0$. If $\alpha\neq\beta$,  then $Der(\mathbf{A})=0$ whereas if $\alpha = \beta$, then $Der(\mathbf{A})$ suggests that the set of all
$\emph{f}\in L(\mathbf{A})$ is a derivation of $\mathbf{A}$ if and only if $e_{2}, e_{4}, e_{5}\in ker(f)$ and $\emph{f}(e_{1}), \emph{f}(e_{3})\in < e_{1} - e_{3} >$.
\\\\In the third case, we have $\mathbb{S}_{1}=\{1,4\}$, $\mathbb{S}_{2}=\{2,4\}$,
 $\mathbb{S}_{3}=\{3\}$, $\mathbb{S}_{4}=\{1,2,4,5\}$, and $\mathbb{S}_{5}=\{4,5\}$, and we obtain
$\mathbb{N}_{1}=\{1,4\}$, $\mathbb{N}_{2}=\{2,4\}$, $\mathbb{N}_{3}=\{3\}$, $\mathbb{N}_{4}=\{4\}$, and
$\mathbb{N}_{5}=\{4,5\}$. Therefore, $D(e_{4}),D(e_{3})=0$. Given that
\begin{center}
$|\mathbb{N}_{1}\diagdown  ( \mathbb{N}_{1} \cap \mathbb{N}_{k})|=|\mathbb{N}_{k}\diagdown  ( \mathbb{N}_{1} \cap \mathbb{N}_{k})|=| \mathbb{N}_{k} \cap \mathbb{N}_{1}|=1.$
\end{center}
we have $D(e_{k})=0$, where k = 1, 2, 5. In this case, $Der(\mathbf{A})=0$.\\
\item[Case 8: ] If $\gamma(\mathbf{A})=8$, then  $\mathbf{A}$ can be graphed up to equivalence and isomorphism as follows:
$$
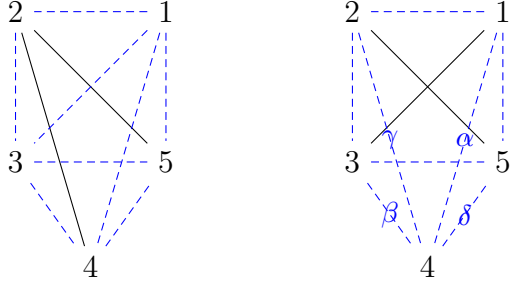
\begin{tikzpicture}\label{100}
[->,>=stealth',shorten >=-5pt,auto,thick,every node/.style={circle,fill=blue!20,draw}]
  \node (n5) at (8,6) {5};
  \node (n2) at (6,8)  {2};
  \node (n4) at (7,4.6)  {4};
  \node (n1) at (8,8) {1};
  \node (n3) at (6,6)  {3};
\path[every node/.style={font=\sffamily\small}]
 (n2) edge node [above] {} (n4)
 (n2) edge node [above] {} (n5);
 \tikzstyle{every node}=[font=\small\itshape]
 \draw[densely dashed,blue] (n1) -- (n2);
 \draw[densely dashed,blue] (n1) -- (n3);
 \draw[densely dashed,blue] (n1) -- (n4);
 \draw[densely dashed,blue] (n1) -- (n5);
 \draw[densely dashed,blue] (n2) -- (n3);
 \draw[densely dashed,blue] (n3) -- (n4);
 \draw[densely dashed,blue] (n3) -- (n5);
 \draw[densely dashed,blue] (n4) -- (n5);
\end{tikzpicture}
\hspace{2cm}
\begin{tikzpicture}\label{100}
[->,>=stealth',shorten >=-5pt,auto,thick,every node/.style={circle,fill=blue!20,draw}]
  \node (n5) at (8,6) {5};
  \node (n2) at (6,8)  {2};
  \node (n4) at (7,4.6)  {4};
  \node (n1) at (8,8) {1};
  \node (n3) at (6,6)  {3};
\path[every node/.style={font=\sffamily\small}]
 (n1) edge node [above] {} (n3)
 (n2) edge node [above] {} (n5);
 \tikzstyle{every node}=[font=\small\itshape]
 \draw[densely dashed,blue] (n1) -- (n2);
 \draw[densely dashed,blue] (n2) -- (n4)node [midway]{$\gamma$};
 \draw[densely dashed,blue] (n1) -- (n4)node [midway]{$\alpha$};
 \draw[densely dashed,blue] (n1) -- (n5);
 \draw[densely dashed,blue] (n2) -- (n3);
 \draw[densely dashed,blue] (n3) -- (n4)node [midway]{$\beta$};
 \draw[densely dashed,blue] (n3) -- (n5);
 \draw[densely dashed,blue] (n4) -- (n5)node [midway]{$\delta$};
\end{tikzpicture}
$$
In the first case, we have $ \mathbb{S}_{r}=\{r\}$, (where r = 1, 3,) $ \mathbb{S}_{2}=\{2,4,5\}$, $ \mathbb{S}_{4}=\{2,4\}$ and
$\mathbb{S}_{5}=\{2,5\}$, and we obtain
$\mathbb{N}_{k}=\{k\}$, (where k = 1, 2, 3)$ \mathbb{N}_{4}=\{2,4\}$ and $\mathbb{N}_{5}=\{2,5\}$. Therefore,
$D(e_{k})=0$ where k = 1, 2, 3. Given that
\begin{center}
$|\mathbb{N}_{4} \diagdown (\mathbb{N}_{5} \cap \mathbb{N}_{4})|=|\mathbb{N}_{4} \cap \mathbb{N}_{5}|=|\mathbb{N}_{4} \diagdown ( \mathbb{N}_{4}\cap \mathbb{N}_{5})|=1.$
\end{center}
 we have  $D(e_{r}) = 0, r = 4, 5$. In this case, $Der(\mathbf{A})=0$.
\\\\ In the second case
, we have  $\mathbb{S}_{1}=\{1,3\}$, $\mathbb{S}_{2}=\{2,5\}$, $\mathbb{S}_{3}=\{1,3\}$, $\mathbb{S}_{4}=\{4\}$, and
$\mathbb{S}_{5}=\{2,5\}$, and we obtain, $\mathbb{N}_{k}=\{1,3\}$, (k = 1, 3)
$\mathbb{N}_{r}=\{2,4\}$, (r = 2, 5) and $\mathbb{N}_{4}=\{4\}$. Therefore,
$D(e_{4})=0$. Following G.1, if $\alpha\neq\gamma$ or $\beta\neq \delta$,  then $D(e_{2}),D(e_{5})=0$. Similarly, if $\alpha\neq\beta$ or $\gamma \neq \delta$, we have $D(e_{1})=D(e_{3})=0$. In this case, $Der(\mathbf{A})=0$.\\

If $\alpha\neq \beta$, then $\emph{f}\in \emph{L}(\mathbf{A})$ can be treated as a derivation if and only if $e_{1},e_{3}\in ker(\emph{f})$ and $\emph{f}(e_{2}),\emph{f}(e_{5})\in <e_{2}-e_{5}>$. Similarly, if $\alpha\neq\gamma$, then $e_{2},e_{5}\in ker(\emph{f})$ and $\emph{f}(e_{1}),\emph{f}(e_{3})\in <e_{1}-e_{3}>$. If $\alpha=\beta\neq\gamma=\delta$, then $\emph{f}\in L(\mathbf{A})$ can be viewed as a derivation of $\mathbf{A}$ if and only if $e_{4}\in ker(f)$ and
\begin{center}
$\emph{f}(e_{1}),\emph{f}(e_{3})\in <e_{1}-e_{3}>$ and $\emph{f}(e_{2}),\emph{f}(e_{5})\in <e_{2}-e_{5}>$.\\
\end{center}
\item[Case 9:]If $\gamma(\mathbf{A})=9$ then
 $\mathbf{A}$ up to equivalence and isomorphism  as follows:\\
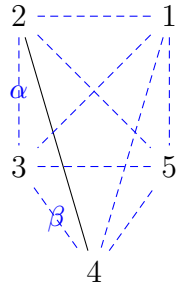
\begin{center}
\begin{tikzpicture}\label{35}
[->,>=stealth',shorten >=-5pt,auto,thick,every node/.style={circle,fill=blue!20,draw}]
  \node (n5) at (8,6) {5};
  \node (n2) at (6,8)  {2};
  \node (n4) at (7,4.6)  {4};
  \node (n1) at (8,8) {1};
  \node (n3) at (6,6)  {3};
\path[every node/.style={font=\sffamily\small}]
 (n2) edge node [above] {} (n4);
 \tikzstyle{every node}=[font=\small\itshape]
 \draw[densely dashed,blue] (n2) -- (n3)node [midway]{$\alpha$};
 \draw[densely dashed,blue] (n3) -- (n4)node [midway]{$\beta$};
 \draw[densely dashed,blue] (n1) -- (n2);
 \draw[densely dashed,blue] (n1) -- (n3);
 \draw[densely dashed,blue] (n1) -- (n4);
 \draw[densely dashed,blue] (n1) -- (n5);
 \draw[densely dashed,blue] (n3) -- (n5);
 \draw[densely dashed,blue] (n4) -- (n5);
 \draw[densely dashed,blue] (n2) -- (n5);
\end{tikzpicture}

\end{center}
In  this case, we have $\mathbb{S}_{k}=\{k\}$ (k = 1, 3, 5) and $\mathbb{S}_{r}=\{2,4\}$, r = 2, 4, and we obtain, $ \mathbb{N}_{k}=\{k\}$ (k = 1, 5, 3)$  \mathbb{N}_{k}=\{2,4\}$ (k = 2, 4), Therefore, $D(e_{k})=0$, where k= 1, 3, 5. Following G.1, if $\alpha\neq \beta$, then $e_{4},e_{2}\in ker(\emph{f})$. Therefore, $Der(\mathbf{A})=0$, and if $\alpha =\beta$, then $\emph{f}\in L(\mathbf{A})$ can be treated as a derivation of $\mathbf{A}$ if and only if  $e_{s}\in ker(f)$, where s= 1, 3, 5, and  $\emph{f}(e_{2}),\emph{f}(e_{4})\in <e_{2}-e_{4}>$.	
\end{itemize}
\end{proof}

\end{document}